\documentclass[11pt]{amsart}
\usepackage{amsmath,amssymb,amsthm}
\usepackage[hidelinks]{hyperref}
\usepackage{graphicx}
\usepackage{tikz}  
\usepackage{pgfplots}
\usepackage{fullpage}
\usepackage[capitalize]{cleveref}
\usepackage{mathrsfs}
\usepackage{nicefrac}

\crefname{enumi}{}{}
\crefname{equation}{}{}

\newtheorem{theorem}{Theorem}[section]

\newtheorem{lemma}[theorem]{Lemma}
\newtheorem{proposition}[theorem]{Proposition}

{\theoremstyle{definition}

\newtheorem{remark}[theorem]{Remark}

}

\newcommand{\R}{\mathbb{R}}
\newcommand{\E}{\mathbb{E}}

\newcommand{\supp}{\operatorname{supp}}

\newcommand{\includegraphicspage}[2][0.3\textwidth]{\begin{minipage}{#1}\includegraphics[width=\textwidth]{#2}\end{minipage}}

\makeatletter
\def\@tocline#1#2#3#4#5#6#7{\relax
\ifnum #1>\c@tocdepth 
\else
\par \addpenalty\@secpenalty\addvspace{#2}%
\begingroup \hyphenpenalty\@M
\@ifempty{#4}{%
\@tempdima\csname r@tocindent\number#1\endcsname\relax
}{%
\@tempdima#4\relax
}%
\parindent\z@ \leftskip#3\relax \advance\leftskip\@tempdima\relax
\rightskip\@pnumwidth plus4em \parfillskip-\@pnumwidth
#5\leavevmode\hskip-\@tempdima
\ifcase #1
\or\or \hskip 1em \or \hskip 2em \else \hskip 3em \fi%
#6\nobreak\relax
\dotfill\hbox to\@pnumwidth{\@tocpagenum{#7}}\par
\nobreak
\endgroup
\fi}
\makeatother

\newcommand{\noncontentsline}[3]{}
\newcommand{\tocless}[2]{\bgroup\let\addcontentsline=\noncontentsline#1{#2}\egroup}

\begin{document}

\title{Non-parametric estimation of non-linear diffusion coefficient in parabolic SPDEs}

\author[Andersson]{Martin Andersson}
\address{Martin Andersson,
  Department of Mathematics,
  Uppsala University,
  S-751 06 Uppsala,
  Sweden}
\author[Avelin]{Benny Avelin}
\address{Benny Avelin,
  Department of Mathematics,
  Uppsala University,
  S-751 06 Uppsala,
  Sweden}
\author[Garino]{Valentin Garino}
\address{Valentin Garino,
  Department of Mathematics,
  Uppsala University,
  S-751 06 Uppsala,
  Sweden}
\author[Ilmonen]{Pauliina Ilmonen}
\address{Pauliina Ilmonen,
  Department of Mathematics and Systems Analysis,
  Aalto University School of Science,
  00076 Aalto,
  Finland}
\author[Viitasaari]{Lauri Viitasaari}
\address{Lauri Viitasaari,
  Department of Information and Service Management, Aalto University School of Business, 00076 Aalto, Finland}

\begin{abstract}
  In this article, we introduce a novel non-parametric predictor, based on conditional expectation, for the unknown diffusion coefficient function $\sigma$ in the stochastic partial differential equation $Lu = \sigma(u)\dot{W}$, where $L$ is a parabolic second order differential operator and $\dot{W}$ is a suitable Gaussian noise. We prove consistency and derive an upper bound for the error in the $L^p$ norm, in terms of discretization and smoothening parameters $h$ and $\varepsilon$. We illustrate the applicability of the approach and the role of the parameters with several interesting numerical examples.
\end{abstract}

\maketitle


\tableofcontents

\section{Introduction}
In this article, we consider the estimation of the unknown diffusion function $\sigma$ from the observed solution $u$ to the equation $Lu = \sigma(u) \dot{W}$, where $\dot{W}$ is a Gaussian noise and $L$ is a suitable parabolic differential operator. As a prototypical case, $L = \partial_t - \Delta$, corresponding to the stochastic heat equation. Stochastic partial differential equations serve as models for complex systems described by partial differential equations that are affected by randomness. These random elements can represent natural random fluctuations arising from the underlying phenomena itself or, for example, measurement errors. In our context, the ``size'' of these fluctuations depend non-linearly on the solution $u$ itself through the function $\sigma(u)$, and estimation of $\sigma$ then corresponds to assessing how large are the random fluctuations that affect the system.

Stochastic partial differential equations (SPDEs) have received a considerable amount of attention in the literature in recent years. To overcome the non-differentiability, one can understand the underlying equation $Lu = \sigma(u)\dot{W}$ via integration in the Dalang-Walsh sense~\cite{dalang1999extending,walsh1984intro}, in which case one can obtain existence and uniqueness of the so-called mild solutions, see~\cite{avelin2021existence,dalang1999extending,walsh1984intro}. While SPDEs are already well-studied, most of the literature focuses on theoretical properties. Studies related to statistical problems are more scarce. In particular, the estimation of coefficient functions have received attention only relatively recently. See~\cite{Cialenco2} for a survey on the topic. The case of constant $\sigma$ and parametric special form of $L$ is studied, e.g., in~\cite{Bossert,Tonaki}, in which (constant) parameters are estimated via spectral approach. For other approaches to the estimation of the constant diffusion coefficient $\sigma$, see~\cite{Bossert-Bibinger}. Estimation problems in the case of non-constant $\sigma$, either estimation of the quantities related to the diffusion $\sigma$ or other quantities related to $L$, are studied in~\cite{Altmeyer2,Almeyer,Bibinger,Cialenco}. However, in these cases $\sigma$ is allowed to be a function of $t$ only, i.e., $\sigma= \sigma(t)$. In this case the solution is Gaussian field as well, allowing for more detailed analysis. For a related literature, see also a recent article~\cite{Almeyer2} and references therein.

To the best of our knowledge, the present article is the first article studying non-parametric estimation of the $\sigma$ that depends on the solution $u$ itself. That is, we consider $\sigma = \sigma(u)$ and our aim is to estimate the unknown function $\sigma$ from observed solutions $u$. Our approach is based on estimating a novel non-parametric predictor, built using conditional expectations. The predictor is a conditional expectation of integrals (over $\varepsilon$-balls) of the discretized operator $L^h u$. We prove the consistency of our predictor and quantify the $L^p$-norm error in terms of the chosen parameters $\varepsilon$ and $h$. For the operator $L$, we adopt the setting of~\cite{avelin2021existence} and consider general parabolic operators $L$ with non-constant coefficients. Such operators admit smooth enough fundamental solutions that, in general, are non-symmetric. Fortunately, one can obtain heat-kernel type bounds for the fundamental solutions and their derivatives, allowing to deduce computations to the heat-kernel level. Our numerical experiments show good performance of our approach if the parameters $\varepsilon$ and $h$ are chosen appropriately.

The rest of the article is organized as follows. In \cref{sec:main} we present and discuss our main results. The estimation approach is examined numerically in \cref{sec:numerical}. Technical preliminaries and auxiliary bounds are presented in \cref{sec:preliminaries}, while the proofs of the main results are presented in \cref{sec:main-proof,sec:white-proof}. We end the paper with \cref{sec:additional_plots} containing additional discussion on numerical experiments.

\section{A consistent predictor of the diffusion function}
\label{sec:main}
Let $T>0$ be fixed. For $x\in\mathbb{R}^d$ and $t\in [0,T]$, we consider the equation
\begin{equation}\label{SPDE-1}
  \begin{cases}
    -Lu(x,t)+\sigma(u(x,t))\dot W(x,t)=0 \\
    u(\cdot,0)=u_0(x),
  \end{cases}
\end{equation}
where $\dot{W}$ is either a Gaussian space-time white noise and $d=1$, or $\dot{W}$ is a Gaussian field that is only white in time and the spatial correlations $\gamma(x-y)$ are given by the Riesz kernel $\gamma(x-y)=\Vert x-y\Vert^{-\beta}$ with $\beta<1$. We suppose that $u_0(x)$ is bounded and Lipschitz, and the coefficient
$\sigma: \mathbb{R}\rightarrow \mathbb{R}$ is a non-negative $M$-Lipschitz function. Finally, we assume that the operator $L$ is given by
\begin{equation}\label{operator}
  L:=\partial_t-\sum_{i,j=1}^d a_{ij}(x,t)\partial_{x_ix_j}-\sum_{i=1}^db_{i}(x,t)\partial_{x_i},
\end{equation}
where $a_{ij},b_i \in \mathcal{C}^{1,1}_b$, i.e.~both $a_{ij}$ and $b_i$ are bounded and  continuously differentiable in both $t$ and $x$ with bounded partial derivatives for $i,j=1,\ldots,d$, and the matrix $(a_{ij})$ is uniformly elliptic.
It is well-known, see, e.g.,~\cite{avelin2021existence,dalang1999extending}, that under these assumptions, \cref{SPDE-1} has a unique weak solution (in the sense of \cref{eq:solution-def}). In the spatially colored case we note that $\beta<\min(d,2)$ is sufficient for the existence and uniqueness of the solution. However, in our setting, as we consider the estimation of $\sigma$, we need the stronger requirement $\beta<1$, to obtain the convergence of the predictor.

A prototypical example of $L$ is
\[
  L = \partial_t - \frac{1}{2}\Delta
\]
corresponding to the Stochastic Heat Equation (SHE). More generally, our assumptions on the operator $L$ ensure that the fundamental solution (also called the Green kernel) associated with \cref{SPDE-1} is sufficiently smooth, and the fundamental solution and its derivatives can be bounded by the corresponding quantities of the heat kernel; cf.~\cref{lma:heat-bounds}.

Our aim is to estimate, non-parametrically, the coefficient function $\sigma$ that corresponds to the ``size of fluctuations'' in \cref{SPDE-1}. We do not assume any specific functional form for $\sigma$. However, due to the symmetry of Gaussian distribution, we can only recover $\sigma$ up to a sign, and hence we assume non-negativity for $\sigma$ to be identifiable.

Consider now constructing a predictor for the unknown $\sigma$ at a given point $u(x_0,t_0)$. Assume first that $\dot{W}$ is a smooth and nicely behaving function. Then, squaring \cref{SPDE-1} leads to
\[
  [(Lu)(x_0,t_0)]^2 = \sigma^2(u(x_0,t_0))[\dot{W}(x_0,t_0)]^2.
\]
Now taking conditional expectation given $u(x_0,t_0)$, this leads to
\[
  [(Lu)(x_0,t_0)]^2 = \sigma^2(u(x_0,t_0))\mathbb{E}[\dot{W}(x_0,t_0)]^2
\]
from which $\sigma^2$ at $u(x_0,t_0)$ can be recovered. However, in our generalized setting, $\dot{W}$ is not a smooth and nicely behaving function. It exists only as a generalized (random) function, and thus the pointwise expectation $\mathbb{E}[\dot{W}(x_0,t_0)]^2$ does not even exist. On top of that, non-smoothness of $\dot{W}$ implies that the solution $u$ itself is not smooth, and hence one cannot apply the operator $L$ directly. Our approach is to make the above computation rigorous by smoothening the problem. Precise heuristic derivations are given in \cref{subsec:heuristic}. We replace the operator $L$  with its finite difference approximation $L^h$ given by
\begin{align}\label{DefLh}
  L^h:f(x,t) \rightarrow & (\mathcal{D}_t^{h^2} f)(x,t)-\sum_{i,j=1}^d a_{ij}(x,t)(\mathcal D^{2,h}_{ij}f)(x,t)+\sum_{i=1}^d b_i(x,t)(\mathcal D^{1,h}_i f)(x,t),\notag \\
  =: & (\mathcal{D}_t^{h^2} f)(x,t)-(\mathcal{S}^h f)(x,t),
\end{align}
where
\begin{align*}
  (\mathcal{D}_t^{h^2} f)(x,t)= & \frac{f(x,t+h^2)-f(x,t)}{h^2}, \\
  (\mathcal D^{1,h}_{i} f)(x,t)= & \frac{f(x+he_i,t)-f(x,t)}{h}   \\
  (\mathcal D^{2,h}_{ij} f)(x,t)= & D^{1,h}_i D^{1,h}_j f(x,t)
\end{align*}
are finite difference approximations for the derivatives, and $e_i$ denote the standard basis vectors. The field $\dot{W}$ is smoothened
by integrating over the region $\mathcal{B}_{x_0,\varepsilon} \times [t_0,t_0+\varepsilon]$, where $\mathcal{B}_{x_0,\varepsilon}$ is the $d$-dimensional $x_0$-centered ball of radius $\varepsilon$ with volume
\begin{equation}\label{ballVolume}
  V(\varepsilon,d)=\frac{\pi^\frac{d}{2}}{\Gamma\left(\frac{d}{2}+1\right)}\varepsilon^d.
\end{equation}
This gives us the predictor
\begin{multline} \label{mainPredictor}
  \widehat{\sigma}^2_{\varepsilon,h}(u(x_0,t_0);x_0,t_0):=
  \\
  \mathbb{E}_{u(x_0,t_0)}\left [\left(\frac{1}{\widetilde{m}(\varepsilon,h)}\int\limits_{\mathcal{B}_{x_0,\varepsilon}}\int\limits_{t_0}^{t_0+\varepsilon}
  \left(L^hu(y,s)-\int\limits_{\mathbb{R}^d}L^h \Gamma(y,s;z,0)u_0(z)dz\right)dsdy\right)^2\right ].
\end{multline}
Here $\Gamma$ is the fundamental solution of the associated partial differential equation (cf.~\cref{subsec:PDE}), $\widetilde{m}(\varepsilon,h)$ is a suitable normalizing sequence, and $\mathbb{E}_{u(x_0,t_0)}$ is the conditional expectation given $u(x_0,t_0)$ that can then be estimated using data, cf. \cref{sec:numerical}. In our results, we choose $\varepsilon = h^{\varrho}$ for $\varrho\in(0,1)$. Set $I_h(t_0) = [t_0,t_0+\varepsilon+h^2]$ and $I_{h,r}(t_0) = [(r-h^2) \vee t_0, r\wedge (t_0+\varepsilon)]$.
We define
\begin{align} \label{eq:tilde_A}
  \widetilde{A}(z,r) := \iint\limits_{I_{h,r}(t_0) \times \mathcal{B}_{x_0,\varepsilon} } \Gamma(y,s+h^2;z,r) dy ds,
\end{align}
and the normalizing sequence in the spatially correlated case is defined by
\begin{align} \label{eq:m1}
  \widetilde{m}^2(\varepsilon,h) := m^2(h) := \frac{1}{h^4} \iiint\limits_{I_h(t_0) \times \mathbb{R}^{2d} } \widetilde{A}(z_1,r) \widetilde{A}(z_2,r)\gamma(z_1-z_2)dz_1 dz_2 dr,
\end{align}
where $\gamma(x) = \Vert x\Vert^{-\beta}$. Then we have $m(h) \sim h^{\varrho(2d-\beta+1)/2}$, see \cref{prop:A}.
Our main result is the following.
\begin{theorem}\label{main}
  Suppose $\dot{W}$ is white in time and spatial correlations are given by $\gamma(x-y) = \Vert x-y\Vert^{-\beta}$ for $\beta<1$. Let $h>0$, set $\varepsilon = h^{\varrho}$ for $\varrho = \frac{8}{12-\beta}$, and let $m(h)$ be given by \cref{eq:m1}. Then, for all $(x_0,t_0)\in\mathbb{R}^d\times [0,T]$, all $p\in [1,\infty)$, and for any $\kappa \in\left(0,\frac{2(2-\beta)}{12-\beta}\right)$, there exists a constant $K= K(M,p,\beta,d,\kappa)$ such that
  \begin{equation}\label{mainEq}
    \E \left [ \left \lvert \widehat{\sigma}^2_{\varepsilon,h}(u(x_0,t_0);x_0,t_0)-\sigma^2(u(x_0,t_0))\right \rvert^p \right ]^{\frac{1}{p}} \leq K h^{\kappa}.
  \end{equation}
\end{theorem}
\begin{remark}
  We remark that the constant $K$ depends only on $M,p,\beta,d,\kappa$ and the fundamental solution $\Gamma$, and not on, e.g., $h$ and $\varepsilon$. The constant $K$ grows without a limit as $\kappa$ approaches $\frac{2(2-\beta)}{12-\beta}$. We also remark that one can obtain the rate for any $\varepsilon = h^\varrho$ with arbitrary $\varrho \in(0,1)$, see \cref{prop:A,prop:R}. The above formulation follows from optimizing the choice of $\varrho$.
\end{remark}

In the space-time white noise case with $d=1$, the normalizing sequence is given by
\begin{align} \label{eq:m2}
  \widetilde{m}^2(\varepsilon,h) := \hat{m}^2(h) := \frac{1}{h^4} \iiint\limits_{I_h(t_0) \times \mathbb{R} } \widetilde{A}^2(z,r) dz dr
\end{align}
that now satisfies $\hat{m}(h) \sim h^{\varrho}$. In this case we obtain the following.
\begin{theorem}
  \label{main2}
  Suppose $\dot{W}$ is white in time and in space. Let $h>0$, set $\varepsilon = h^{\varrho}$ for $\varrho = \nicefrac{8}{9}$, and let $\hat{m}(h)$ be given by \cref{eq:m2}. Then, for all $(x_0,t_0)\in\mathbb{R}\times [0,T]$, all $p\in [1,\infty)$, and for any $\kappa \in\left(0,\nicefrac{2}{9}\right)$, there exists a constant $K= K(M,p,\kappa)$ such that
  \begin{equation}\label{mainEq2}
    \E \left [ \left \lvert \widehat{\sigma}^2_{\varepsilon,h}(u(x_0,t_0);x_0,t_0)-\sigma^2(u(x_0,t_0))\right \rvert^p \right ]^{\frac{1}{p}}
    \leq K h^{\kappa}.
  \end{equation}
\end{theorem}

\section{Numerical experiments} \label{sec:numerical}

In this section, we present some numerical experiments to illustrate the performance of our predictor. For our experiments, we consider the SHE driven by space-time white noise. This is a well-studied model in the field of stochastic partial differential equations. The SHE we consider is given by
\begin{equation} \label{eq:SHE}
  \begin{aligned}
    \partial_t u(x,t) & = \frac{1}{2} \Delta u(x,t) + \sigma(u(x,t)) \xi(x,t), \\
    u(x,0)            & = 6,                                                   \\
    u(0,t)            & = u(L,t) = 0,
  \end{aligned}
\end{equation}
where $\xi$ is a space-time white noise. The considered functions $\sigma$ are listed in \cref{tab:sigmas}. The bases for choosing these functions is to see how the size and the shape of the function $\sigma$ affect the performance of our predictor. The main prototype form for $\sigma$ is $\sigma_3$, which is symmetric around $x=2$, has a corner at $x=2$, two smooth peaks, and very flat regions where the values are small, see \cref{fig:sigma_example}. The other functions are variations of this function with different amplitudes and oscillations. Functions $\sigma^2_4$ and $\sigma^2_5$ have the
same Lipschitz constant as $\sigma_3^2$.
Function $\sigma_6$ has small values and small oscillations. Finally, $\sigma_7$  has small values and large oscillations.

\begin{table}[htbp]
  \centering
  \caption{Functions $\sigma$ that are used in the simulated examples.}
  \begin{tabular}{c|c}
    $\sigma$                                                     & Characteristics                           \\
    \hline
    $\sigma_1(x)  = 0.1$                                         & small constant function                   \\
    $\sigma_2(x) = 2$                                            & large constant function                   \\
    $\sigma_3(x) = \frac{1}{8} \exp(\sin(4|x-2|)) $              & nonlinear function and small values       \\
    $\sigma_4(x)  = \frac{1}{32} \exp(\sin(4|x-2|)) + 1 $        & large values and small Lipschitz constant \\
    $\sigma_5(x) = \frac{1}{8} \exp(\sin(\frac{4}{5}|x-2|)) + 1$ & large values and small Lipschitz constant \\
    $\sigma_6(x)  = \frac{1}{8} \exp(\sin(\frac{4}{10}|x-2|))$   & small values and small oscillations       \\
    $\sigma_7(x)  = \frac{1}{8} \exp(\sin(13|x-2|))$             & small values and large oscillations
  \end{tabular}
  \label{tab:sigmas}
\end{table}
Our theoretical result, see \cref{main}, covers equations on the whole space. However, for computational reasons, our simulations are restricted  to bounded domains with zero Dirichlet boundary conditions.

\subsection{Numerical approach}
We use the following numerical approach to solve \cref{eq:SHE}. First, we discretize the spatial domain $[0,L]$ with points $N_x$, which implies a mesh size of $\Delta x = \nicefrac{L}{N_x}$. We then use a standard finite difference scheme from~\cite{gyongy1998latticeI,gyongy1999latticeII} to approximate \cref{SPDE-1}. This leads to a system of $N_x$ coupled SDEs, $\mathbf{u}(t) := [u_1(t),\ldots,u_{N_x}(t)]$, defined by
\begin{equation} \label{eq:SDE}
  \begin{aligned} 
    d \mathbf{u}(t) & = A \mathbf{u}(t)dt + \Sigma(\mathbf{u}(t)) dW_t, \\
    u_i(0)          & = 6,                                              \\
    u_1(t)          & = u_{N_x}(t) = 0,
  \end{aligned}
\end{equation}
where $W_t$ is a standard Wiener process in $\R^{N_x}$, $A$ is a tri-diagonal matrix of the form
\begin{equation}
  A = \frac{N_x^2}{2}
  \begin{pmatrix}
    0      & 0      & 0      & \cdots & 0      & 0      & 0      \\
    1      & -2     & 1      & \cdots & 0      & 0      & 0      \\
    0      & 1      & -2     & \cdots & 0      & 0      & 0      \\
    \vdots & \vdots & \vdots & \ddots & \vdots & \vdots & \vdots \\
    0      & 0      & 0      & \cdots & 1      & -2     & 1      \\
    0      & 0      & 0      & \cdots & 0      & 0      & 0
  \end{pmatrix},
\end{equation}
and where $\Sigma(\mathbf{u}(t))$ is a diagonal matrix with first and last diagonal elements equal to zero and the rest of the diagonal elements are given by $(\Sigma(\mathbf{u}(t)))_{ii}=\sigma(u_i(t)) \sqrt{N_x}$. To obtain a numerical solution to \cref{eq:SDE}, we discretize the time interval $[0,T]$ with $N_t$ points and simulate using the classical Euler-Maruyama method. For details on the spatial approximation and convergence rates of our numerical scheme, we refer to~\cite{Anton,gyongy1999latticeII}.

In our experiments, we use the parameter values $L=1$, $T=1$, $N_x = 2^9$, and $N_t = N_x^2 = 2^{18}$. The resulting numerical solution $\mathbf{u}(t)$ is taken as the underlying true solution to the SHE in \cref{eq:SHE}.

\begin{figure}[ht]
  \centering
  \begin{tikzpicture}
    \begin{axis}[
        width=0.7\textwidth,
        height=0.4\textwidth,
        domain=0:4,
        samples=400,
        xlabel={$x$},
        ylabel={$y$},
        ymin=0, ymax=0.15,
        yticklabel style={/pgf/number format/fixed, /pgf/number format/precision=2},
        thin
      ]
      \addplot[black, thin] {1/64 * exp(2*sin((57.2957795131)*4*abs(x-2)))};
    \end{axis}
  \end{tikzpicture}
  \caption{Plot of function $\sigma_3^2$.}
  \label{fig:sigma_example}
\end{figure}
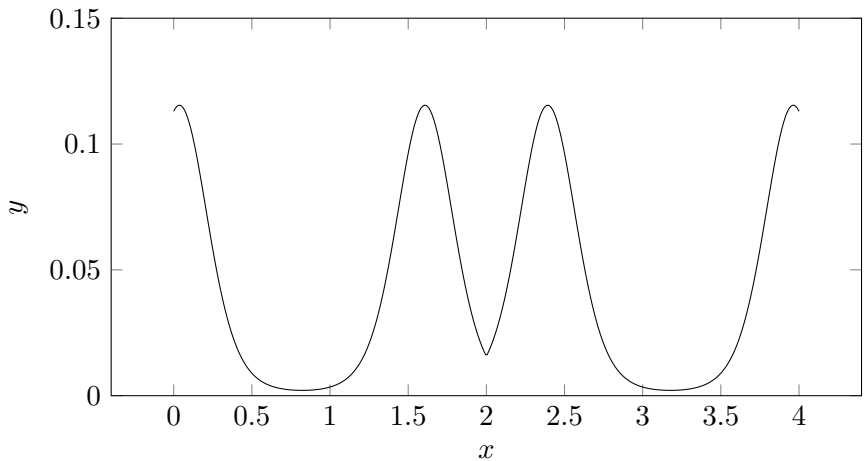

\subsection{Generation of the data-set and calculation of the integral}
In our experiments, we consider the predictor \cref{mainPredictor} with various combinations $h \in \{\nicefrac{2}{N_x}, \nicefrac{4}{N_x}, \nicefrac{8}{N_x}\}$ and $\varepsilon \in \{h, 2h, 4h, 8h\}$.

We next describe a discrete approximation $\widetilde{\sigma}^2_{\varepsilon,h}(x_0,t_0)$ of \cref{mainPredictor} satisfying
\begin{equation} \label{eq:regression_function}
  \widehat{\sigma}^2_{\varepsilon,h}(u(x_0,t_0);x_0,t_0)
  \approx \E_{u(x_0,t_0)}[\widetilde{\sigma}^2_{\varepsilon,h}(x_0,t_0)].
\end{equation}
For this, let $i_0,j_0$ be the discrete coordinates of $(x_0,t_0)$, i.e.~integers $i_0 = \nicefrac{N_x x_0}{L}$ and $j_0 = \nicefrac{N_t t_0}{T}$. With given $h$ and $\varepsilon$, the window $W = [x_0 - \varepsilon,x_0 + \varepsilon] \times [t_0 , t_0 + \varepsilon]$ corresponding to the mesh
\begin{equation}
  \widehat{W} = \{(i,j) \in \mathbb{Z}^2 : |i-i_0| \leq \nicefrac{N_x\varepsilon}{L}, 0 \leq j-j_0 \leq \nicefrac{N_t\varepsilon}{T}\}.
\end{equation}
Then the first part of \cref{mainPredictor} is based on the approximation
\begin{equation} \label{eq:discrete_integral}
  \frac{1}{m(\varepsilon)} \int\limits_{B_\varepsilon} \int\limits_{t_0}^{t_0+\varepsilon} L^h u (x,t) dx dt \approx \frac{1}{\sqrt{2}\varepsilon} \frac{TL}{N_x N_t} \sum_{(i,j) \in \widehat{W}} (L^h \mathbf{u})_i(t_j),
\end{equation}
where $L^h \mathbf{u}(t)$ is defined as
\begin{equation}
  (L^h \mathbf{u})_i(t_j) := \frac{u_i(t_{j+h^2 N_t/T}) - u_i(t_j)}{h^2} - \frac{u_{i-h N_x/L}(t_j) - 2u_i(t_j) + u_{i+h N_x/L}(t)}{h^2}.
\end{equation}
Note that the above quantities are well-defined for all points $(x_0,t_0)$ in the discretized domain $[0,L] \times [0,T]$ that are far enough from the boundaries.
For computing the term $\int\limits_{\mathbb{R}^d}L^h \Gamma(y,s;z,0)u_0(z)dz$ in \cref{mainPredictor}, we calculate
$\hat{\mathbf{u}}$ as the solution to \cref{eq:SDE} with zero noise using the same steps as above. This leads to
\begin{equation} \label{eq:discrete_sigma}
  \widetilde{\sigma}^2_{\varepsilon,h}(x_0,t_0) :=  \left [ \frac{1}{\sqrt{2}\varepsilon} \frac{TL}{N_x N_t} \sum_{(i,j) \in \widehat{W}} ((L^h \mathbf{u})_i(t_j)-(L^h \hat{\mathbf{u}})_i(t_j)) \right ]^2.
\end{equation}

In our experiments, for each used combination of $h$ and $\varepsilon$, we produce 100 realizations of the process $\mathbf{u}$. Each realization has 10000 points $(x_0,t_0)$ in the discretized domain $[0,L] \times [0,T]$ (for which the expression~\cref{eq:discrete_sigma} is defined).

\subsection{Estimation of $\widehat{\sigma}^2_{\varepsilon,h}$} 
Motivated by \cref{eq:regression_function}, we estimate the conditional expectation \\
$\E_{u(x_0,t_0)} [\widetilde{\sigma}^2_{\varepsilon,h}(x_0,t_0)]$ using kernel regression on the simulated data-set. The bandwidth of the kernel is chosen to be the same for all combinations of $h$ and $\varepsilon$. The results of the regression are shown in \cref{fig:plot_sigma_function_1,fig:plot_sigma_function_2,fig:grid_sigma_functions,fig:plot_sigma_functions_matrix,fig:grid_sigma_functions_epsilon_order} for each of the functions $\sigma_1,\ldots,\sigma_7$. Functions $\sigma_1$ and $\sigma_2$ were experimented  with $h= \nicefrac{2}{N_x}$ and $\varepsilon \in \{h, 2h, 4h\}$. The findings are displayed in \cref{fig:plot_sigma_function_1,fig:plot_sigma_function_2}. The functions $\sigma_4,\ldots,\sigma_7$ were experimented with $h= \nicefrac{2}{N_x}$. The used $\varepsilon$ with the results are displayed in  \cref{fig:plot_sigma_functions_matrix}. The prototypical function $\sigma_3$ was tested with all combinations of $h \in \{\nicefrac{2}{N_x}, \nicefrac{4}{N_x}, \nicefrac{8}{N_x}\}$ and $\varepsilon \in \{h, 2h, 4h, 8h\}$. The results are shown in \cref{fig:grid_sigma_functions,fig:grid_sigma_functions_epsilon_order}. \cref{fig:grid_sigma_functions_epsilon_order} displays same plots as \cref{fig:grid_sigma_functions}. The plots are reordered such that each row corresponds to a fixed value of $\varepsilon$ and each column corresponds to a fixed value of $h$. This makes it easier to assess the effect of $h$ for a fixed $\varepsilon$ and the effect of $\varepsilon$ for a fixed $h$. Additional analysis for $\sigma_3$ is provided in \cref{sec:additional_plots}.

The light blue regions in \cref{fig:plot_sigma_function_1,fig:plot_sigma_function_2,fig:grid_sigma_functions,fig:plot_sigma_functions_matrix,fig:grid_sigma_functions_epsilon_order} are the approximate 95\% prediction intervals. This illustrates the uncertainty of the regression and its relation to how large values the function $\sigma^2$ attains. The red dashed line is the true function $\sigma^2$ and the blue line is the estimated function $\hat{\sigma}^2$. The gray scattered points are simulated observations.
Note that most of the simulated values of $u$ are concentrated around zero. While this is expected, as the solution to the heat equation decays rapidly, it makes the estimation challenging for large values.
The starting point of each numerical solution is 6, but as $u$ decreases rapidly, there are only very few simulated observations near the initial value 6. The plots are thus presented only for $u\leq 4$.

\begin{figure}[ht]
  \centering
  \begin{minipage}{0.32\textwidth}
    \centering
    \includegraphics[width=\textwidth]{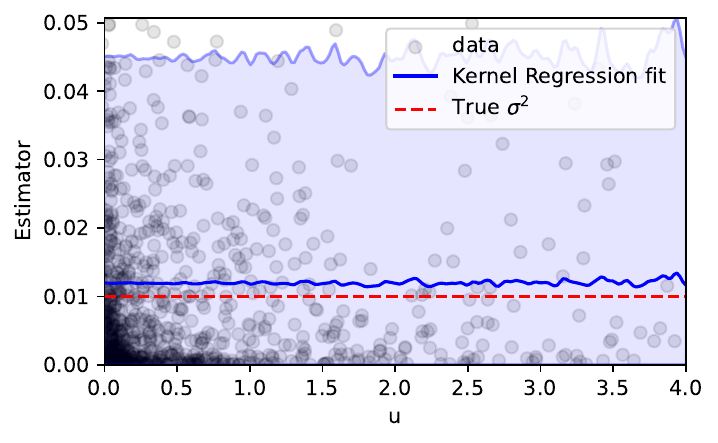}
    $\varepsilon = h$, $h = \nicefrac{2}{N_x}$
  \end{minipage}
  \begin{minipage}{0.32\textwidth}
    \centering
    \includegraphics[width=\textwidth]{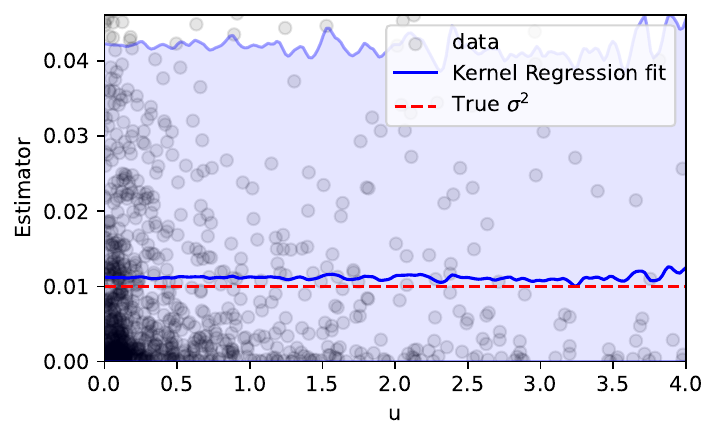}
    $\varepsilon = 2h$, $h = \nicefrac{2}{N_x}$
  \end{minipage}
  \begin{minipage}{0.32\textwidth}
    \centering
    \includegraphics[width=\textwidth]{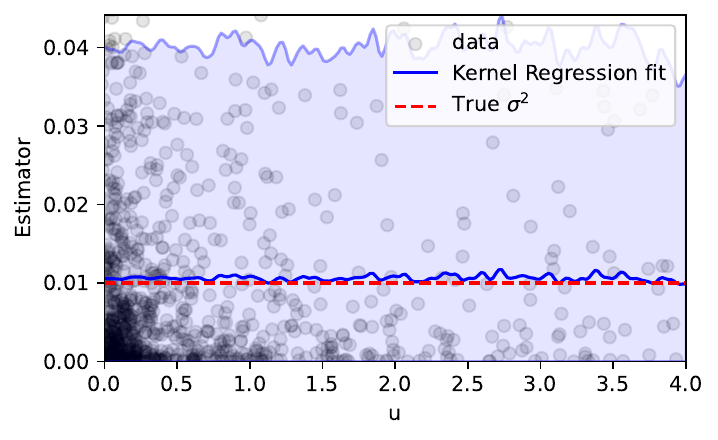}
    $\varepsilon = 4h$, $h = \nicefrac{2}{N_x}$
  \end{minipage}
  \caption{Results for function $\sigma^2_1$ for different combinations of $h$ and $\varepsilon$.}
  \label{fig:plot_sigma_function_1}
\end{figure}

\begin{figure}[ht]
  \centering
  \begin{minipage}{0.32\textwidth}
    \centering
    \includegraphics[width=\textwidth]{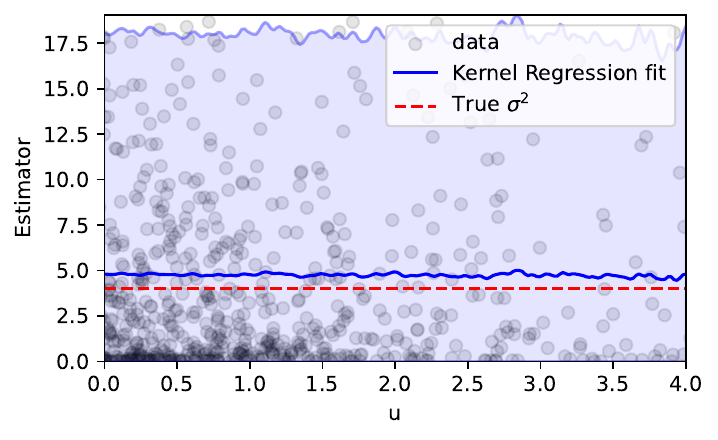}
    $\varepsilon = h$, $h = \nicefrac{2}{N_x}$
  \end{minipage}
  \begin{minipage}{0.32\textwidth}
    \centering
    \includegraphics[width=\textwidth]{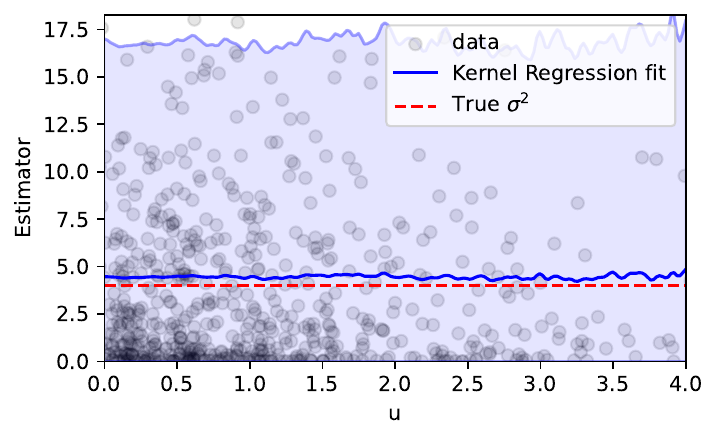}
    $\varepsilon = 2h$, $h = \nicefrac{2}{N_x}$
  \end{minipage}
  \begin{minipage}{0.32\textwidth}
    \centering
    \includegraphics[width=\textwidth]{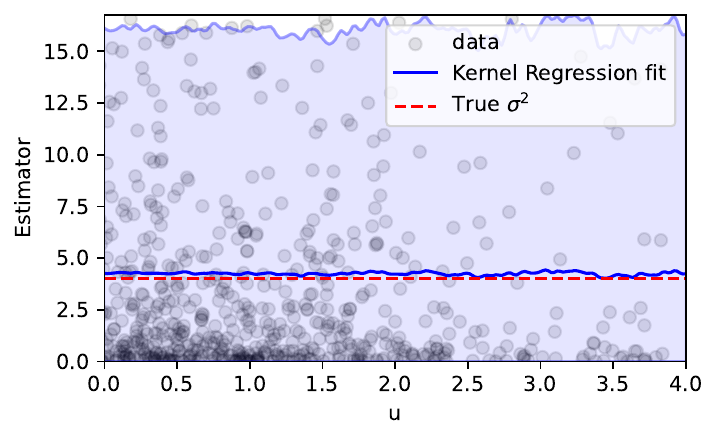}
    $\varepsilon = 4h$, $h = \nicefrac{2}{N_x}$
  \end{minipage}
  \caption{Results for function $\sigma^2_2$ for different combinations of $h$ and $\varepsilon$.}
  \label{fig:plot_sigma_function_2}
\end{figure}

\begin{table}[ht]
  \centering
  \begin{tabular}{c|ccc}
    $\varepsilon \backslash h$ & $\nicefrac{2}{N_x}$ & $\nicefrac{4}{N_x}$ & $\nicefrac{8}{N_x}$ \\ 
    \hline
    $h$ & \includegraphicspage[0.29\textwidth]{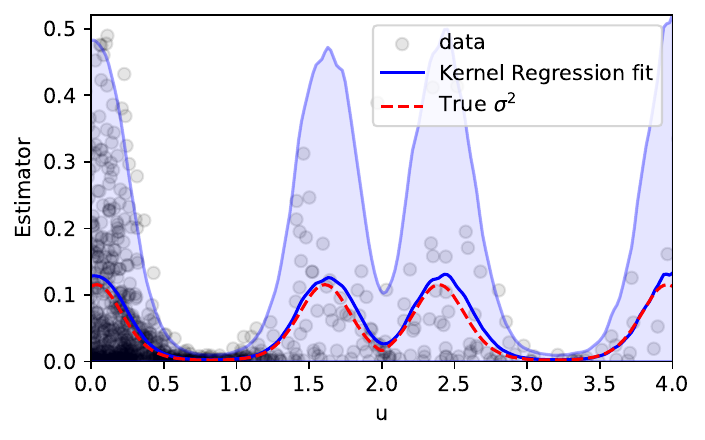} & \includegraphicspage[0.29\textwidth]{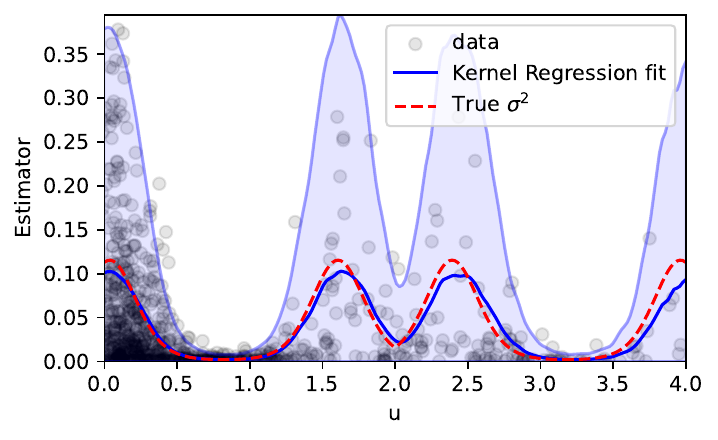} & \includegraphicspage[0.29\textwidth]{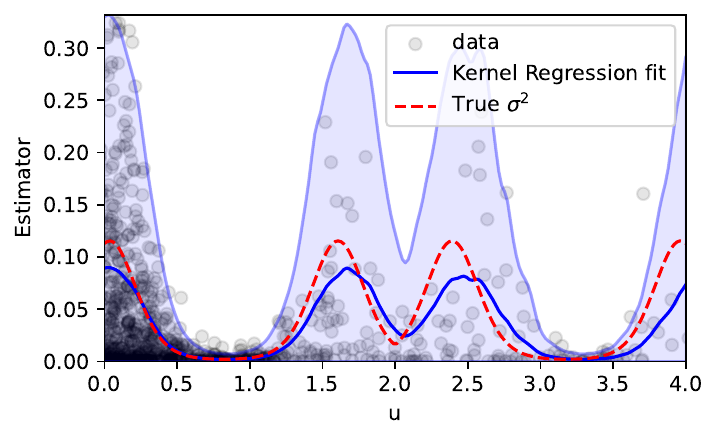} \\
    $2h$ & \includegraphicspage[0.29\textwidth]{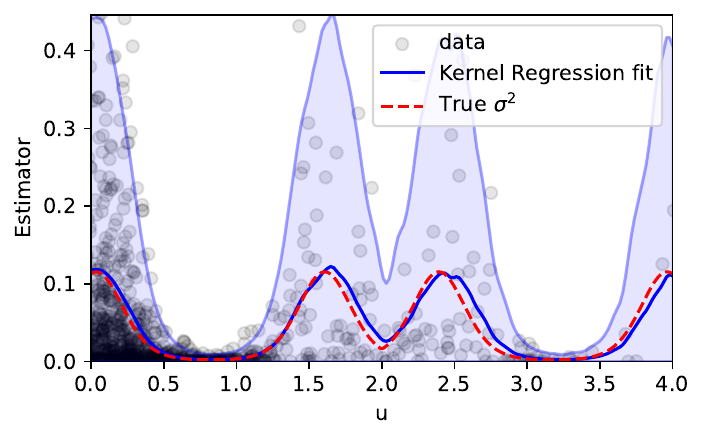} & \includegraphicspage[0.29\textwidth]{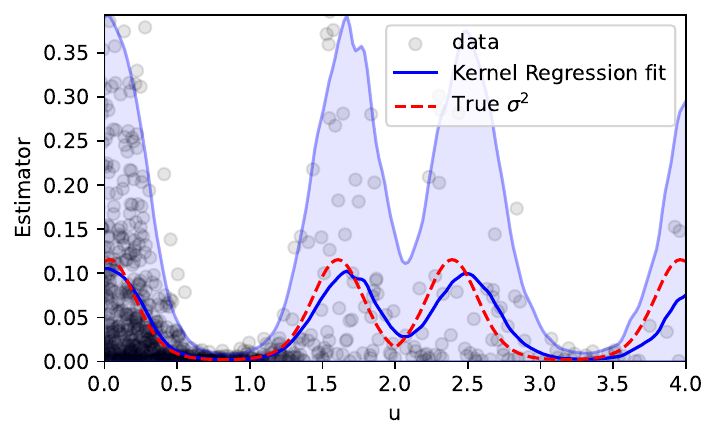} & \includegraphicspage[0.29\textwidth]{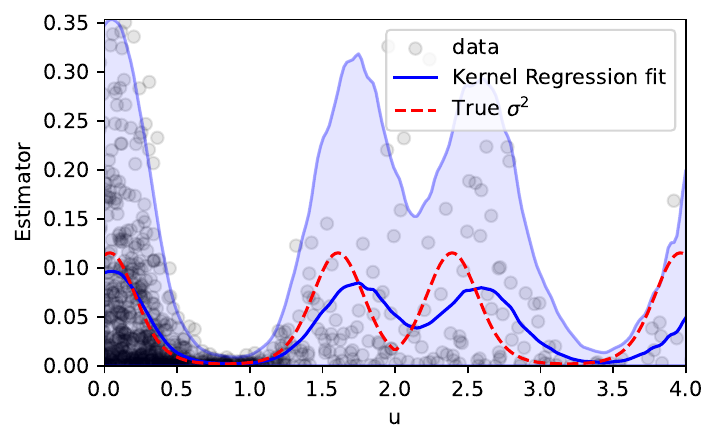} \\
    $4h$ & \includegraphicspage[0.29\textwidth]{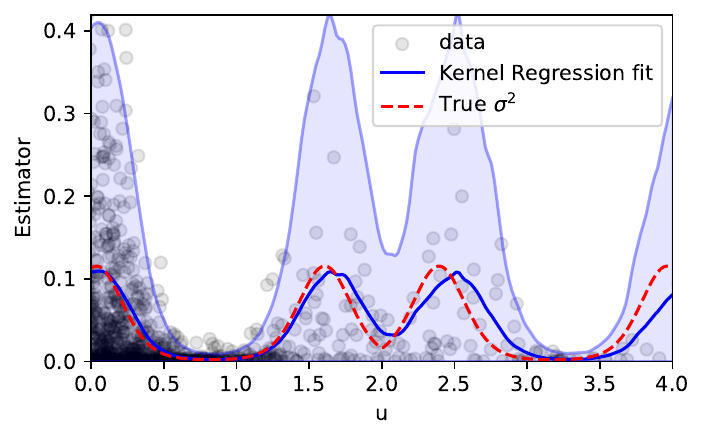} & \includegraphicspage[0.29\textwidth]{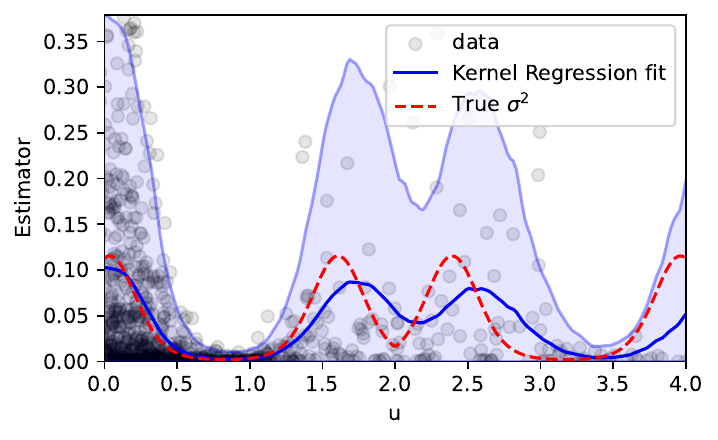} & \includegraphicspage[0.29\textwidth]{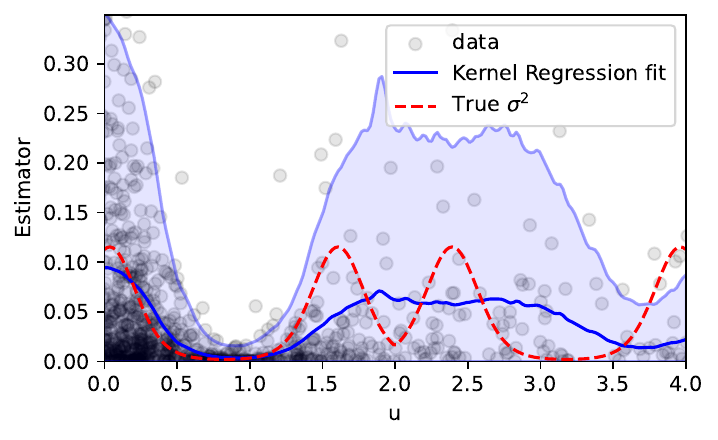} \\
    $8h$ & \includegraphicspage[0.29\textwidth]{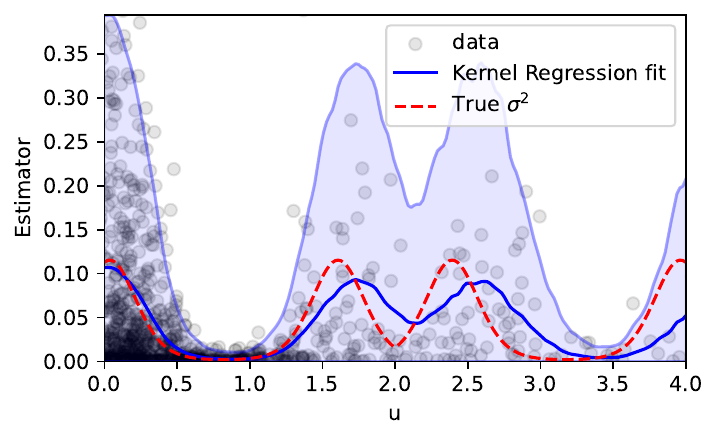} & \includegraphicspage[0.29\textwidth]{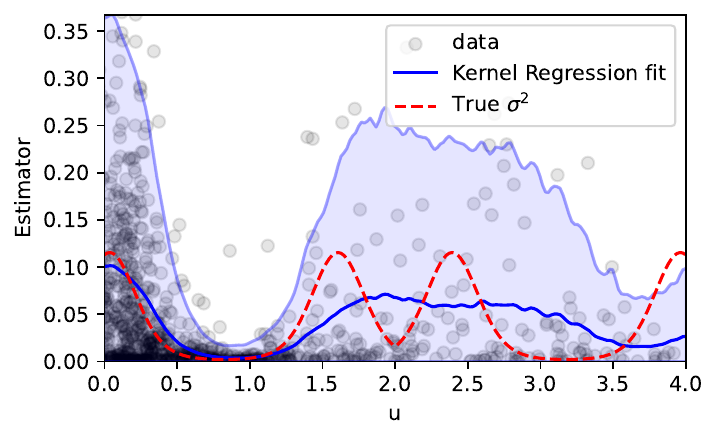} & \includegraphicspage[0.29\textwidth]{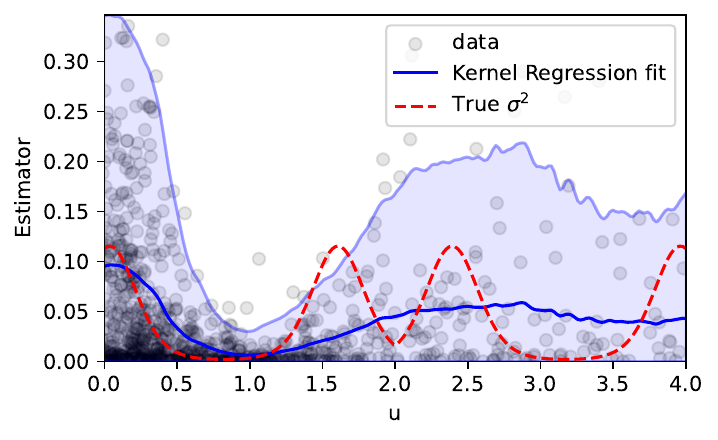} \\
  \end{tabular}
  \caption{Results for function $\sigma^2_3$ for different combinations of $h$ and $\varepsilon$.}
  \label{fig:grid_sigma_functions}
\end{table}

\begin{table}[ht]
  \centering
  \begin{tabular}{c|ccc}
    $\varepsilon \backslash h$ & $\nicefrac{2}{N_x}$ & $\nicefrac{4}{N_x}$ & $\nicefrac{8}{N_x}$ \\ 
    \hline
    $\nicefrac{4}{N_x}$ & \includegraphicspage[0.29\textwidth]{images/plot_sigma_function_3_2_2.pdf} & \includegraphicspage[0.29\textwidth]{images/plot_sigma_function_3_1_4.pdf} & \includegraphicspage[0.29\textwidth]{images/plot_sigma_function_3_1_8.pdf} \\
    $\nicefrac{8}{N_x}$ & \includegraphicspage[0.29\textwidth]{images/plot_sigma_function_3_4_2.pdf} & \includegraphicspage[0.29\textwidth]{images/plot_sigma_function_3_2_4.pdf} & \includegraphicspage[0.29\textwidth]{images/plot_sigma_function_3_1_8.pdf} \\
    $\nicefrac{16}{N_x}$ & \includegraphicspage[0.29\textwidth]{images/plot_sigma_function_3_8_2.pdf} & \includegraphicspage[0.29\textwidth]{images/plot_sigma_function_3_4_4.pdf} & \includegraphicspage[0.29\textwidth]{images/plot_sigma_function_3_2_8.pdf} \\
    $\nicefrac{32}{N_x}$ & & \includegraphicspage[0.29\textwidth]{images/plot_sigma_function_3_8_4.pdf} & \includegraphicspage[0.29\textwidth]{images/plot_sigma_function_3_4_8.pdf} \\
  \end{tabular}
  \caption{Results for function $\sigma^2_3$ for fixed values of $\varepsilon$ (rows) and for fixed values of $h$ (columns).}
  \label{fig:grid_sigma_functions_epsilon_order}
\end{table}

\begin{table}[ht]
  \centering
  \begin{tabular}{c|ccc}
    $\varepsilon \backslash h$ & $\nicefrac{2}{N_x}$ & $\nicefrac{4}{N_x}$ & $\nicefrac{8}{N_x}$ \\ \hline
    $\nicefrac{4}{N_x}$ & \boxed{0.0070}  & 0.0076 & \\
    $\nicefrac{8}{N_x}$ & \boxed{0.0127}  & 0.0129 & 0.0142\\
    $\nicefrac{16}{N_x}$ & 0.0206 & \boxed{0.0205}  & 0.0207 \\
    $\nicefrac{32}{N_x}$ &        & \boxed{0.0282}  & 0.0288
  \end{tabular}
  \caption{The $L^1$-errors (over values $0\leq u\leq 4$) for different values of $\varepsilon$ and $h$ in the case of $\sigma^2_3$. The smallest values of each row are highlighted.}
  \label{tab:l1_errors}
\end{table}

\begin{figure}[ht]
  \centering
  \begin{minipage}{0.45\textwidth}
    \centering
    \includegraphics[width=\textwidth]{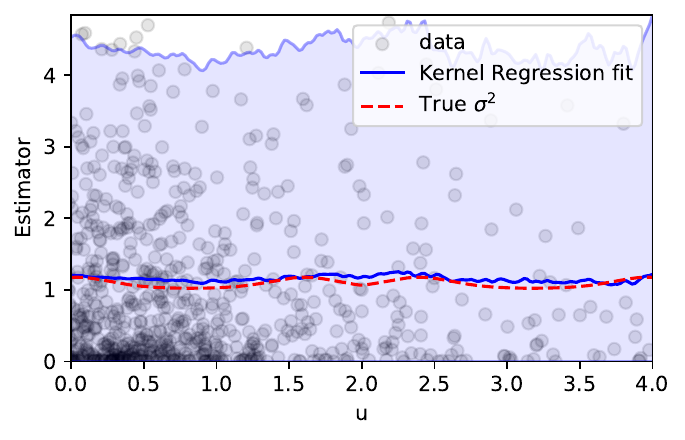}
    $\sigma_4(x)$, $\varepsilon = 4h$, $h = \nicefrac{2}{N_x}$
  \end{minipage}
  \begin{minipage}{0.45\textwidth}
    \centering
    \includegraphics[width=\textwidth]{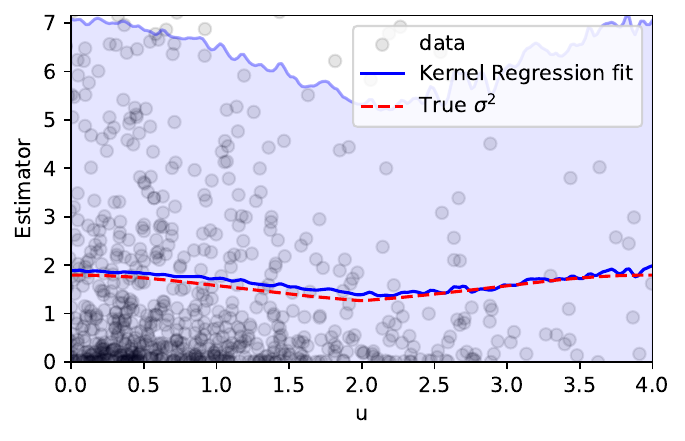}
    $\sigma_5(x)$, $\varepsilon = 4h$, $h = \nicefrac{2}{N_x}$
  \end{minipage}

  \vspace{0.5cm}

  \begin{minipage}{0.45\textwidth}
    \centering
    \includegraphics[width=\textwidth]{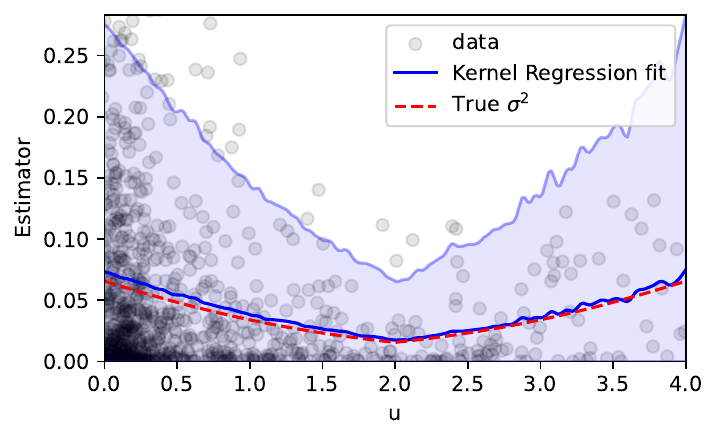}
    $\sigma_6(x)$, $\varepsilon = 2h$, $h = \nicefrac{2}{N_x}$
  \end{minipage}
  \begin{minipage}{0.45\textwidth}
    \centering
    \includegraphics[width=\textwidth]{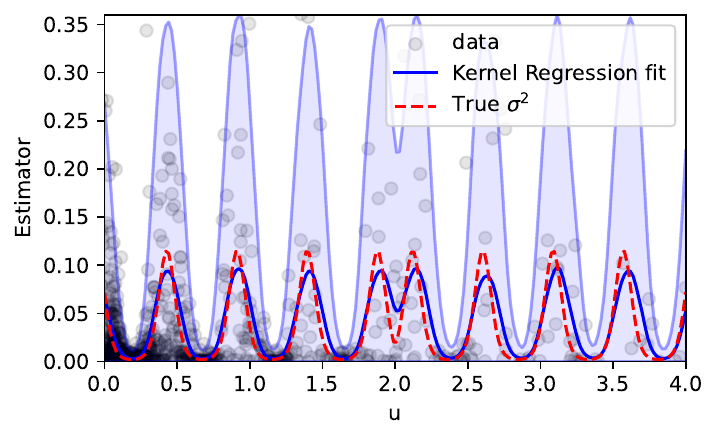}
    $\sigma_7(x)$, $\varepsilon = h$, $h = \nicefrac{2}{N_x}$
  \end{minipage}

  \caption{Results for functions $\sigma^2_4,\ldots,\sigma^2_7$ for different combinations of $h$ and $\varepsilon$.}
  \label{fig:plot_sigma_functions_matrix}
\end{figure}

\clearpage
\subsection{Discussion}

Our numerical experiments reveal several interesting aspects of the predictor's behavior.

In the case of constant $\sigma$, the bias seems to be always positive. Moreover, in this case,
increasing $\varepsilon$ leads to bias reduction,  see \cref{fig:plot_sigma_function_1,fig:plot_sigma_function_2}. Explanation for this can be seen by considering the decomposition of the predictor~\cref{mainPredictor}, see~\cref{eq:A,eq:R}. In particular, the term $R_{y,s,h}$ gives rise to a positive bias. Moreover, by examining our proof, we observe that some error terms have a scale $(h/\varepsilon) \sigma^2(u(x_0,t_0))$ that decreases as $\varepsilon$ increases. The bias proportional to $(h/\varepsilon) \sigma^2(u(x_0,t_0))$ can be seen from \cref{fig:plot_sigma_function_1,fig:plot_sigma_function_2} as the relative bias (with respect to $\sigma^2$) is similar for $\sigma_1$ and $\sigma_2$.
As $\varepsilon$ increases, we also observe a decrease in the width of the 95\% prediction intervals (note that the scale of the $y$-axis changes). This indicates that the predictor becomes more stable.

In the case of prototypical non-constant function $\sigma_3$, the effect of varying $h$ while keeping $\varepsilon = k h$ fixed is displayed in \cref{fig:grid_sigma_functions}. Within each column, $h$ is a constant and $\varepsilon$ increases down the rows. The variance seems to decrease as $\varepsilon$ increases. For each row, the ratio $\nicefrac{h}{\varepsilon}$ is fixed. By examining the plots within each row, we notice that the variance seems to decrease as $h$ increases. On the other hand, we observe that as $\varepsilon$ increases, the peaks of $\sigma_3^2$ are smoothened and shifted to the right. This might be explained by the strong diffusivity of the heat equation, that could also explain why the effect seems to be weaker closer to zero, since at zero the solution is essentially just noise. This can be compared to the ``ballistic'' versus ``diffusive'' regimes of the classical Ornstein-Uhlenbeck process. Further discussion of this phenomenon can be found in \cref{sec:additional_plots}. \cref{fig:grid_sigma_functions_epsilon_order} illustrates the effect of fixing $\varepsilon$ and increasing $h$ within each row. This behavior closely resembles the case where $h$ is fixed and $\varepsilon$ increases, suggesting that both parameters play similar roles in the estimation problem. In \cref{tab:l1_errors} we see the effect of varying $\varepsilon$ and $h$ on the integrated $L^1$ error given by
\[
  \int\limits_0^4 \left \lvert \mathbb{E}_u\left [ \widetilde{\sigma}^2_{\varepsilon,h} \right ] - \sigma_3^2(u) \right \rvert du.
\]
It is clear that for fixed $\varepsilon$, it is optimal to choose $h$ as small as possible. Similarly, for fixed $h$, one should choose $\varepsilon$ as small as possible.

The behavior of $\sigma_7$ in \cref{fig:plot_sigma_functions_matrix} showcases an example where biases to different directions may lead to a fairly accurate estimate.
There both $\varepsilon$ and $h$ are small. This could lead to a large variance and positive bias as in \cref{fig:plot_sigma_function_1,fig:plot_sigma_function_2}. However, for a highly oscillating $\sigma$, the smoothing effect seems to be strong enough to compensate this overestimation.

\section{Preliminary bounds and auxiliary results}\label{sec:preliminaries}

In this section, we present the preliminary bounds and auxiliary results needed for the proof of \cref{main}. The proof utilizes several auxiliary results and technical lemmas that are also provided in this section. In the sequel, $c$, $C$, and $K$ denote unimportant constants that may or may not depend on a set of parameters, and that may vary from line to line.

\subsection{White-colored noise and Walsh integrals}\label{sec:prl}
In this section we review the necessary preliminaries on stochastic integration related to \cref{SPDE-1}. For details and further reading, we refer to~\cite{dalang1999extending,walsh1984intro}.

For a given covariance kernel $\gamma$, we define the inner product for $\varphi,\psi\in\mathcal C_c^{\infty}(\mathbb{R}^d\times\mathbb{R}_+)$ through the formula
\begin{equation}
  \label{eq:covariance}
  \langle \varphi,\psi\rangle_{\gamma}:=\int\limits_{\mathbb{R}_+}\int\limits_{\mathbb{R}^d}\int\limits_{\mathbb{R}^d}\varphi(x,s)\psi(y,s)\gamma(y-x)dydxds,
\end{equation}
where (with a slight abuse of notation)
\[
  \int\limits_{\mathbb{R}^d}\int\limits_{\mathbb{R}^d}\varphi(x,s)\psi(y,s)\gamma(y-x)dydx:=\int\limits_{\mathbb{R}^d}\varphi(y,s)\left(\psi(s,\cdot)*\gamma\right)(y)dy,
\]
and
$\ast$ denotes the convolution in the space variable. In our setting, we have $\gamma(x) = \Vert x\Vert^{-\beta}$ or $\gamma(x)=\delta(x)$ (the dirac delta).

Consider the centered Gaussian family indexed by $\mathcal C_c^{\infty}(\mathbb{R}^d\times\mathbb{R}_+)$ through the covariance
\[
  \mathbb{E}[W(\varphi)W(\psi)]=\langle\varphi,\psi\rangle_{\gamma}.
\]
Now, by a standard isometry argument, it is possible to extend the family $W$ to the closure $\mathcal H$ of $\mathcal C_c^{\infty}(\mathbb{R}^d\times\mathbb{R}_+)$ with respect to $\langle,\rangle_{\gamma}$. This allows us to define the noise $\dot{W}$ having (formally) the covariance structure
\[
  \mathbb{E}[\dot{W}(x,s)\dot{W}(y,t)]=\delta(t-s)\gamma(x-y).
\]

For all $t\geq 0$, let $\mathscr{F}_t$ be the $\sigma$-field generated by the family $\left\{W(\varphi)\right\}$, where $\varphi$ has its support in $\mathbb{R}^d \times [0,t]$. Let $\{X(y,s), (y,s)\in\mathbb{R}^d \times \R_+\}$ be a jointly measurable process, adapted with respect to $\mathscr{F}_t$, satisfying
\[
  \mathbb{E}[\|X\|^2_{\gamma}]<\infty.
\]
Now, the integral
\[
  \int\limits_0^T\int\limits_{\mathbb{R}^d}X(y,s)W(dy,ds)
\]
is well-defined and satisfies the isometry
\begin{equation}
  \mathbb{E}\left [\left(\int\limits_{0}^{\infty}\int\limits_{\mathbb{R}^d}X(y,s)W(dy,ds)\right)^2\right ] = \int\limits_0^T\int\limits_{\mathbb{R}^d}\int\limits_{\mathbb{R}^d}\mathbb{E}[X(x,s)X(y,s)]\gamma(x-y)dxdyds. \label{Isometry}
\end{equation}
Moreover, the process $t\mapsto\int\limits_0^t\int\limits_{\mathbb{R}^d}X(y,s)W(dy,ds)$ is also adapted. Since $t\rightarrow W_{t}(\varphi )$ is a martingale for every $\varphi \in\mathcal B_b(\mathbb{R}^d)$ (bounded Borel measurable functions), the following known results are obtained.
\begin{lemma}
  \label{lemma:conditional}
  Let $s<t$ and let $X$ be an adapted square integrable process. Then, for every square integrable $F \in \mathscr{F}_s$, we have
  \begin{align}
     & \mathbb{E}\left [F\int\limits_s^t\int\limits_{\mathbb{R}^d}X(y,u)W(dy,du)\right ] = \mathbb{E}\left [F\left.\int\limits_s^t\int\limits_{\mathbb{R}^d}X(y,u)W(dy,du)\right|\mathscr{F}_s\right ]=0.\label{conditionalExp}
  \end{align}
  Moreover, we have
  \begin{align*}
    \mathbb{E}\left [\left.\left(\int\limits_{0}^t\int\limits_{\mathbb{R}^d}X(y,u)W(dy,du)\right)^2\right|\mathscr{F}_s\right ] = \int\limits_s^t\int\limits_{\mathbb{R}^d}\int\limits_{\mathbb{R}^d}\mathbb{E}[X(z_1,u)X(z_2,u)|\mathscr{F}_s]\gamma(z_1-z_2)dz_1dz_2du,
  \end{align*}
  and, for every $p\geq 1$, there exists $C_p>0$ such that
  \begin{align}
    \mathbb{E}\left [\left|\int\limits_s^t\int\limits_{\mathbb{R}^d}X(y,u)W(dy,du)\right|^p\right ] \leq C_p \left|\int\limits_{s}^t\int\limits_{\mathbb{R}^d}\int\limits_{\mathbb{R}^d}\mathbb{E}[|X(z_1,u)X(z_2,u)|^{\frac{p}{2}}]^{\frac{2}{p}}\gamma(z_1-z_2)dz_1dz_2du\right|^\frac{p}{2}\label{eq:BDG}.
  \end{align}
\end{lemma}
In the literature, \cref{eq:BDG} is called the Burkholder-Davis-Gundy inequality.

\subsection{Auxiliary estimates for SPDEs and for the fundamental solution}
\label{subsec:PDE}
Consider the equation
\begin{equation}\label{SPDE}
  Lu(x,t)=\sigma(u(x,t))\dot{W}(x,t)
\end{equation}
with initial condition $u_0(0,x)=u_0$, where $L$ is a second order (linear) differential operator
and $W$ is a white-colored noise. Let $\Gamma$ be the fundamental solution of the associated PDE $Lu=0$. Then, assuming that $\sigma$ is Lipschitz continuous and that $\Gamma$ verifies
\[
  \sup_{x\in \mathbb{R}^d} \int\limits_{\mathbb{R}^d}\Gamma(x,t;\xi,\tau) d\xi < C_T, \quad 0 \leq \tau < t \leq T,
\]
as well as the condition
\[
  \int\limits_{\mathbb{R}^d}\frac{\|\xi\|^{\beta-d}}{(1+f(\|\xi\|))}d\xi<\infty,
\]
where $f$ is a function which satisfies $|\Gamma(x,t;\xi,\tau)|\leq \mathcal{F}^{-1}(e^{-|t-\tau|f})(x-\xi)$, with $\mathcal{F}^{-1}$ denoting the inverse Fourier transform. It is proved in~\cite{dalang1999extending} (as a particular case of Theorem 13) and~\cite{avelin2021existence} that the equation \cref{SPDE} admits a unique (weak) solution satisfying
\begin{equation}
  \label{eq:solution-def}
  u(x,t)=\int\limits_{\mathbb{R}^d}\Gamma(x,t;y,0)u_0(y)dy+\int\limits_0^t\int\limits_{\mathbb{R}^d}\Gamma(x,t;y,s)\sigma(u(y,s))W(dy,ds),
\end{equation}
where the second integral is to be understood in the sense defined above. It was shown in~\cite{avelin2021existence} that under our assumptions, the two above conditions are verified with $f(x)=\|x\|^2$ leading to the standard assumption $\beta<\min\{2,d\}$.
In the particular case of the \textit{stochastic heat equation}, we have
\begin{equation}\label{SHEsolution}
  u(x,t)=\int\limits_{\mathbb{R}^d}p_t(x-y)u_0(y)dy+\int\limits_0^t\int\limits_{\mathbb{R}^d}p_{t-s}(x-y)\sigma(u(y,s))W(dy,ds),
\end{equation}
where $p$ is the \textit{heat kernel}
\begin{equation}
  p_t(x):=\frac{1}{(\sqrt{2\pi t})^d}e^{-\frac{\|x\|_2^2}{2t}}.
\end{equation}

We recall the following two lemmas for estimating the fundamental solution $\Gamma$ and its derivatives.
\begin{lemma}{\cite[p. 253]{friedman2008partial}} \label{lem:Ia}
  Let
  \[
    I_a = \int\limits_{\R^d} \frac{1}{[(t-s)(s-\tau)]^{d/2}} \exp \left [ -a \frac{|x-\eta|^2}{t-s} - a\frac{|\eta-\xi|^2}{s-\tau} \right ] d\eta,
  \]
  where $\tau < s < t$ and $a > 0$.
  Then, for any $\varepsilon > 0$, there is a constant $M$ such that
  \[
    I_a \leq \frac{M}{(t-\tau)^{d/2}} \exp \left [ -a(1-\varepsilon) \frac{|x-\xi|^2}{t-\tau} \right ].
  \]
\end{lemma}
\begin{lemma}
  \label{lma:heat-bounds}
  Assume that the derivatives $D^{\boldsymbol{\alpha}}_x a_{ij}(x,t)$ and $D^{\boldsymbol{\alpha}}_x b_j(x,t)$, for multi-indices $0 \leq |{\boldsymbol{\alpha}}| \leq r$, where $r$ is a positive integer, exist and are bounded continuous functions of $(x,t)$ in $\R^d \times [0,T]$.
  Let $\Gamma(x,t;\xi,\tau)$ be the fundamental solution to $L$ given by \cref{operator}. Then for any multi-index $\mathbf{m}$ with $|\mathbf{m}|\leq 2$ we have
  \begin{align*}
    |D^{\mathbf{m}+{\boldsymbol{\alpha}}}_x \Gamma(x,t;\xi,\tau)| \leq \frac{C}{(t-\tau)^{(d+|{\boldsymbol{\alpha}}|+|\mathbf{m}|)/2}} \exp \left [ - C \frac{|x-\xi|^2}{t-\tau} \right ],
  \end{align*}
  for a constant $C$.
  If $\partial^k_t a_{ij}$ and $\partial^k_t b_j$ exist, and are bounded and continuous functions of $(x,t)$ in $\R^d \times [0,T]$, then
  \begin{align*}
    |\partial^k_t \Gamma(x,t;\xi,\tau)| \leq \frac{C}{(t-\tau)^{(d+2k)/2}} \exp \left [ - C \frac{|x-\xi|^2}{t-\tau} \right ].
  \end{align*}
\end{lemma}

\begin{proof}
  The proof below is classical, although not easy to find. For the reader's convenience, we provide a sketch using scattered parts of the construction in~\cite{friedman2008partial}, mostly built on Chapter 9.

  The construction of $\Gamma$ goes through the parametrix method by Levi. Denote
  \begin{align*}
    v^{y,s}(x-\xi) = & \sum_{i,j=1}^d a_{ij}(y,s)(x_i-\xi_i)(x_j-\xi_j), \\
    w^{y,s}(x-\xi,t;\tau) = & \frac{1}{(t-\tau)^{d/2}} \exp \left [ -\frac{v^{y,s}(x-\xi)}{4(t-\tau)} \right ], \\
    Z(x-\xi,t;\xi,\tau) = & C(\xi,\tau) w^{\xi,\tau}(x-\xi,t;\tau), \quad \textrm{and} \\
    C(x,t) = & (2\sqrt{\pi})^{-d} \sqrt{|\det(a_{ij})(x,t)|}.
  \end{align*}
  Now $Z(x-\xi,t;\xi,\tau)$ solves the constant coefficient equation with no lower order terms given by
  \[
    L_0 Z (x,t) \equiv \mathrm{Tr}(A(\xi,\tau) D_x^2 Z) - \partial_t Z = 0,
  \]
  where $D_x^2 Z$ is the Hessian matrix of $Z$ with respect to $x$ and we can represent the fundamental solution $\Gamma$ to $L$ as
  \begin{align} \label{eq:Parametrix}
    \Gamma(x,t;\xi,\tau) = Z(x-\xi,t; \xi,\tau) + \int\limits_\tau^t \int\limits_{\R^d} Z(x-\eta,t;\eta,s)\Phi(\eta,s;\xi,\tau) d\eta ds,
  \end{align}
  where $\Phi$ is determined such that $L\Gamma = 0$. In particular, $\Phi$ satisfies
  \begin{align*}
    \Phi(x,t;\xi,\tau) = L Z(x-\xi,t;\xi,\tau) + \int\limits_\tau^t \int\limits_{\R^d} LZ(x-\eta,t;\eta,s) \Phi(\eta,s; \xi,\tau) d\eta ds,
  \end{align*}
  where
  \begin{multline*}
    L Z(x-\xi,t;\xi,\tau) = \sum_{i,j=1}^d (a_{ij}(x,t) - a_{ij}(\xi,\tau))\frac{\partial^2}{\partial x_i \partial x_j} Z(x-\xi,t;\xi,\tau)
    \\
    + \sum_{i=1}^d b_i(x,t) \frac{\partial}{\partial x_i} Z(x-\xi,t;\xi,\tau).
  \end{multline*}
  In the rest of the proof, we use the notation
  $LZ(x-\xi,t;\xi,\tau) = \mathcal{K}(x,t;\xi,\tau)$ for the \emph{correction kernel}.
  The correction kernel measures how well the constant coefficient solution solves the variable coefficient equation.
  It can be shown, see~\cite{friedman2008partial}, that $\Phi$ can be given as
  \begin{align*}
    \Phi(x,t;\eta,s) = \sum_{m=1}^\infty \mathcal{K}_m(x,t;\eta,s),
  \end{align*}
  where $\mathcal{K}_1 = \mathcal{K}$ and, recursively,
  \begin{align*}
    \mathcal{K}_{m+1}(x,t;\xi,\tau) = \int\limits_\tau^t \int\limits_{\R^d} \mathcal{K}(x,t;\eta,s) \mathcal{K}_m(\eta,s;\xi,\tau) d\eta ds.
  \end{align*}
  We first show that for some constant $C$, we have
  \begin{align} \label{eq:DhPhi}
    |D^{\boldsymbol{\alpha}}_x \Phi(x,t;\xi,\tau)| \leq \frac{C}{(t-\tau)^{(d+|{\boldsymbol{\alpha}}|)/2}} \exp \left [ -C \frac{|x-\xi|^2}{t-\tau} \right ].
  \end{align}
  We first recall that if $D^{\boldsymbol{\alpha}}_x a_{ij}$ exist and are bounded continuous functions of $(x,t)$, then the correction kernel $\mathcal{K}$ satisfies see~\cite{friedman2008partial},
  \begin{align} \label{eq:Kbound}
    |D^{\boldsymbol{\alpha}}_x \mathcal{K}(x,t;\xi,\tau)| \leq \frac{C}{(t-\tau)^{(d+1+|{\boldsymbol{\alpha}}|)/2}} \exp \left [ -C \frac{|x-\xi|^2}{t-\tau} \right ],
  \end{align}
  see~\cite{friedman2008partial}.
  Now, following the proof of~\cite[Theorem 7, p.260]{friedman2008partial}, we write $\mathcal{K}_2$ as
  \begin{align*}
    \mathcal{K}_2(x,t;\xi,\tau) = & \int\limits_{\tau}^{\tau + (t-\tau)/2} \int\limits_{\R^d} \mathcal{K}(x,t;\eta,s)\mathcal{K}(\eta,s;\xi,\tau) d\eta ds
    \\
    & +
    \int\limits_{\tau + (t-\tau)/2}^t \int\limits_{\R^d} \mathcal{K}(x,t;\eta,s)\mathcal{K}(\eta,s;\xi,\tau) d\eta ds
    \\
    =: & \mathcal{K}_{21} + \mathcal{K}_{22}.
  \end{align*}
  In term $\mathcal{K}_{21}$ note that $t-s \geq (t-\tau)/2 > 0$, and thus the integral is absolutely convergent, and we can change the order of integration and differentiation. Using also \cref{eq:Kbound,lem:Ia} gives
  \begin{align*}
    |D^{\boldsymbol{\alpha}}_x \mathcal{K}_{21}(x,t;\xi,\tau)| \leq & C \int\limits_{\tau + (t-\tau)/2}^t \frac{1}{(t-s)^{(1+|{\boldsymbol{\alpha}}|)/2}(s-\tau)^{1/2}} I_C(x,t;\xi,\tau,s) ds
    \\
    \leq&
    C \int\limits_{\tau + (t-\tau)/2}^t \frac{1}{(t-s)^{(1+|{\boldsymbol{\alpha}}|)/2}(s-\tau)^{1/2} } ds \frac{1}{(t-\tau)^{d/2}}\exp \left [-C_1 \frac{|x-\xi|^2}{t-\tau} \right ]
    \\
    \leq&
    \frac{C}{(t-\tau)^{(d+|{\boldsymbol{\alpha}}|)/2}}\exp \left [-C_1 \frac{|x-\xi|^2}{t-\tau} \right ].
  \end{align*}
  Treating term $\mathcal{K}_{22}$ similarly shows that we gain one power of $\sqrt{t-\tau}$ when moving from $\mathcal{K}_1$ to $\mathcal{K}_2$, and repeating the argument recursively for $\mathcal{K}_m$, leads to
  \begin{align*}
    |D^{\boldsymbol{\alpha}}_x \mathcal{K}_{m}(x,t;\xi,\tau)| \leq \frac{B_m}{(t-\tau)^{(d+|{\boldsymbol{\alpha}}|+(2-m))/2}}\exp \left [-C_m \frac{|x-\xi|^2}{t-\tau} \right ],
  \end{align*}
  for constants $B_m$ and $C_m$. It is shown in~\cite[p.251-255]{friedman2008partial} that $\sum_{m=1}^\infty D^{\boldsymbol{\alpha}}_x \mathcal{K}_m(x,t;\xi,\tau)$ is convergent. This leads to \cref{eq:DhPhi}.

  We next show that \cref{eq:DhPhi} gives the desired bound for $\Gamma$. For this, we use the representation formula \cref{eq:Parametrix}. Let $\mathbf{s}$ be a multi-index $0 \leq |\mathbf{s}| \leq r$ then note that since $D^{\mathbf{s}}_y a_{ij}(y,t)$ are bounded continuous functions, we have
  \[
    |D^{\boldsymbol{\beta}}_x D^{\mathbf{s}}_y Z(x-\xi,t;y,\tau)| \leq \frac{B_{\boldsymbol{\beta}}}{(t-\tau)^{(d+|\boldsymbol{\beta}|)/2}} \exp \left [ - b_{\boldsymbol{\beta}} \frac{|x-\xi|^2}{t-\tau}\right ],
  \]
  for any multi-index $\boldsymbol{\beta}$.
  This gives the desired bound for the first term in \cref{eq:Parametrix}. The second term can be handled by interchanging $D^{\boldsymbol{\alpha}}_x$ to $D^{\boldsymbol{\alpha}}_\xi$ when differentiating $Z$,  by applying integration by parts, \cref{eq:DhPhi}, and \cref{lem:Ia}. We have now covered the first claim of the lemma.  To prove estimates on the time derivatives of $\Gamma$, one can use that
  \begin{align*}
    \frac{\partial \Gamma(x,t;\xi,\tau)}{\partial t} = L\Gamma (x,t;\xi,\tau).
  \end{align*}
  Together with the first part, this directly leads to the claimed bound for $\frac{\partial \Gamma(x,t;\xi,\tau)}{\partial t}$. To obtain the claimed estimate for higher order derivatives, one simply differentiates both sides with respect to $t$. This covers the second part of the lemma.
\end{proof}

\subsection{Moment bounds}
In order to prove our main result, we need regularity and moment bounds for the solution $u$, see \cref{lem:moment_bounds_differences} below. Similar results for SHE is proved in~\cite{chen2014holder,sanz2002holder,walsh1984intro}.
The following lemma is a straightforward extension of a result for SHE from~\cite{sanz2002holder} to the case of a general, possibly non-symmetric, fundamental solution $\Gamma$. For the reader's convenience and for the sake of completeness, we present the proof.
\begin{lemma}
  \label{lemma:integral-alpha-rep}
  For $\alpha \in \left(0,\nicefrac{1}{2} -\nicefrac{\beta}{4}\right)$, set
  \begin{align*}
    Y^\alpha(x,t) = \int\limits_{0}^t \int\limits_{\mathbb{R}^d} \Gamma(x,t;y,s) \sigma(u(y,s)) (t-s)^{-\alpha} W(dy,ds).
  \end{align*}
  Then
  \begin{align*}
    \int\limits_0^t \int_{\mathbb{R}^d} \Gamma(x,t;y,s)  \sigma(u(y,s))  W(dy,ds) = C(\alpha) \int_0^t \int_{\mathbb{R}^d} \Gamma(x,t;z,r)(t-r)^{\alpha-1} Y^\alpha(z,r) dr dz,
  \end{align*}
  where $C(\alpha) = \frac{\sin(\pi \alpha)}{\pi}$. Moreover, for any $p\geq 2$, we have
  \begin{equation*}
    \sup_{0 \leq s \leq t} \sup_{x \in \R^d} \mathbb{E}\left [|Y^\alpha(x,s)|^p\right ] < \infty.
  \end{equation*}
\end{lemma}
\begin{proof}
  We have
  \begin{multline*}
    \int_0^t \int_{\mathbb{R}^d} \Gamma(x,t;z,r)(t-r)^{\alpha-1} Y^\alpha(z,r) dr dz
    \\
    = \int_0^t \int_{\mathbb{R}^d} \Gamma(x,t;z,r)(t-r)^{\alpha-1} \int\limits_{0}^r \int\limits_{\mathbb{R}^d} \Gamma(z,r;y,s) \sigma(u(y,s)) (r-s)^{-\alpha} W(dy,ds) dr dz.
  \end{multline*}
  By stochastic Fubini's theorem and the semigroup property of $\Gamma$, here
  \begin{align*}
    = &
    \int\limits_0^t \int\limits_{0}^r \int_{\mathbb{R}^d} \Gamma(x,t;y,s) (t-r)^{\alpha-1} (r-s)^{-\alpha} \sigma(u(y,s))  W(dy,ds) dr
    \\
    = & \int\limits_0^t \left ( \int\limits_{s}^t (t-r)^{\alpha-1} (r-s)^{-\alpha} dr \right ) \int_{\mathbb{R}^d} \Gamma(x,t;y,s)  \sigma(u(y,s))  W(dy,ds)
    \\
    = & \frac{1}{C(\alpha)}\int\limits_0^t \int_{\mathbb{R}^d} \Gamma(x,t;y,s)  \sigma(u(y,s))  W(dy,ds).
  \end{align*}
  This shows the first claim. For the second claim, we observe by arguing as in~\cite{sanz2002holder} (using only the Gaussian upper bounds of \cref{lma:heat-bounds}) that
  \begin{align}
    \label{eq:LpBoundSSS}
    \sup_{0 \leq s \leq t} \sup_{x \in \R^d } \mathbb{E}\left [|Y^\alpha(x,s)|^p\right ] < \infty
  \end{align}
  provided that
  \begin{align}
    \int_{\mathbb{R}^d} \frac{\hat \gamma(\xi)}{(1+\|\xi\|^2)^{1-2\alpha}} d\xi = \int_{\mathbb{R}^d} \frac{\|\xi\|^{\beta-d}}{(1+\|\xi\|^2)^{1-2\alpha}} d\xi \leq C \int_{1}^\infty r^{\beta-3+4\alpha} dr < \infty.
  \end{align}
  This requires $\beta-3+4\alpha < -1$, or equivalently $\alpha < \nicefrac{1}{2}-\nicefrac{\beta}{4}$.
\end{proof}

\begin{lemma} \label{lem:moment_bounds_differences}
  Consider the integral
  \begin{align*}
    I(x,t) = \int\limits_{0}^t \int\limits_{\mathbb{R}^d} \Gamma(x,t;y,s) \sigma(u(y,s)) W(dy,ds)
  \end{align*}
  for a mild solution $u$ of \cref{SPDE-1}. Let $p\geq 2$. Then, for any $\nu \in (0,1-\nicefrac{\beta}{2})$, $h\in\mathbb{R}^d$, and $t\in [0,T]$, we have
  \begin{align*}
    \mathbb{E}\left [|I(x,t)-I(x+h,t)|^p\right ]
     & \leq C |h|^{\nu p},
  \end{align*}
  and, for any $\nu \in (0,\nicefrac{1}{2}-\nicefrac{\beta}{4})$, $h\in (0,T)$, and $t\in [0,T-h]$, we have
  \begin{align*}
    \mathbb{E}\left [|I(x,t+h)-I(x,t)|^p\right ]
     & \leq C_T h^{\nu p}.
  \end{align*}
\end{lemma}
\begin{proof}
  We begin with the spatial differences. Using \cref{lemma:integral-alpha-rep} with $\alpha \in \left(0,\nicefrac{1}{2} -\nicefrac{\beta}{4}\right)$, we obtain
  \begin{align*}
    I(x,t) - I(x+h,t) = C(\alpha)\int\limits_{0}^t \int\limits_{\mathbb{R}^d} \left( \Gamma(x,t;z,r) (t-r)^{\alpha-1} - \Gamma(x+h,t;z,r) (t-r)^{\alpha-1} \right) Y^\alpha(z,r) dr dz .
  \end{align*}
  By Hölder's inequality and \cref{eq:LpBoundSSS}, we have
  \begin{align*}
    \mathbb{E}\left [|I(x,t)-I(x+h,t)|^p\right ]
     & \leq C \left | \int\limits_{0}^t \int\limits_{\mathbb{R}^d} \left( \Gamma(x,t;z,r) (t-r)^{\alpha-1} - \Gamma(x+h,t;z,r) (t-r)^{\alpha-1} \right) dr dz  \right |^p.
  \end{align*}
  Using the mean value theorem and \cref{lma:heat-bounds}, we can write
  \begin{align*}
    |\Gamma(x,t;z,r) - \Gamma(x+h,t;z,r)| \leq
    \frac{C |h|}{(t-r)^{(d+1)/2}} e^{-\frac{|z-\bar x|^2}{C(t-r)}} \leq \frac{C |h|}{(t-r)^{(d+1)/2}} \left ( e^{-\frac{|z-x-h|^2}{C(t-r)}} + e^{-\frac{|z-x|^2}{C(t-r)}} \right ),
  \end{align*}
  where $\bar x $ is a point on the line segment between $x$ and $x+h$.
  Now, since $\int_{\R^d} \Gamma(x,t;z,r) dz = 1$, we obtain, for any $\delta\in(0,1)$, an estimate
  \begin{align*}
    \int\limits_{\mathbb{R}^d} & \left | \Gamma(x,t;z,r) (t-r)^{\alpha-1} - \Gamma(x+h,t;z,r) (t-r)^{\alpha-1} \right | dz \\
    & \leq (t-r)^{\alpha - 1} \left ( \int\limits_{\mathbb{R}^d} \left | \Gamma(x,t;z,r) - \Gamma(x+h,t;z,r) \right | dz \right )^{\delta} \\
    & \leq (t-r)^{\alpha - 1 - \frac{\delta}{2}} |h|^{\delta} \left ( \frac{1}{t^{d/2} }\int\limits_{\mathbb{R}^d} e^{-\frac{|z-x-h|^2}{C(t-r)}} dz + \frac{1}{t^{d/2} }\int\limits_{\mathbb{R}^d} e^{-\frac{|z-x|^2}{C(t-r)}} dz \right )^{\delta} \\
    & \leq C |h|^{\delta} (t-r)^{\alpha - 1 - \frac{\delta}{2}}.
  \end{align*}
  Here
  \begin{align*}
    \int_0^t (t-r)^{\alpha - 1 - \frac{\delta}{2}} dr \leq C t^{\alpha - \frac{\delta}{2}}
  \end{align*}
  whenever $\delta \in (0,2\alpha)$. This leads to
  \begin{align*}
    \mathbb{E}\left [|I(x,t)-I(x+h,t)|^p\right ] \leq C_T |h|^{\delta p}
  \end{align*}
  for any $\delta \in (0,2\alpha)$. Since $\alpha <\nicefrac{1}{2} -\nicefrac{\beta}{4}$, we obtain the first claim. For the second claim, we proceed similarly and write
  \begin{align*}
    I(x,t+h) - I(x,t) = \int\limits_{0}^t \int\limits_{\mathbb{R}^d} \left( \Gamma(x,t+h;z,r) (t+h-r)^{\alpha-1} - \Gamma(x,t;z,r) (t-r)^{\alpha-1} \right) Y^\alpha(z,r) dr dz,
  \end{align*}
  where now we have, by denoting $\hat \Gamma(x,t;z,r) = (t-r)^{\frac{d}{2}} \Gamma(x,t;z,r)$, that
  \begin{align*}
    \int\limits_{\mathbb{R}^d} \left | \hat \Gamma(x,t+h;z,r) (t+h-r)^{\alpha-1-\frac{d}{2}} - \hat \Gamma(x,t;z,r) (t-r)^{\alpha-1-\frac{d}{2}} \right | dz \\
    \leq
    (t+h-r)^{\alpha-1-\frac{d}{2}} \int\limits_{\mathbb{R}^d} \left | \hat \Gamma(x,t+h;z,r) - \hat \Gamma(x,t;z,r) \right | dz
    \\
    +
    \int\limits_{\mathbb{R}^d} \Gamma(x,t;z,r) \left |  (t+h-r)^{\alpha-1-\frac{d}{2}} - (t-r)^{\alpha-1-\frac{d}{2}}  \right | dz.
  \end{align*}
  Using the Gaussian upper bounds given by \cref{lma:heat-bounds} we see that $\hat \Gamma(x,t;z,r)$ is a bounded function. Together with the mean value theorem, this allows us to obtain that, for any $\delta \in (0,1)$, we have
  \begin{align*}
    \hat \Gamma(x,t+h;z,r) - \hat \Gamma(x,t;z,r) \leq C |\hat \Gamma(x,t+h;z,r) - \hat \Gamma(x,t;z,r)|^{\delta} = h^\delta \left | \partial_t \hat \Gamma(x,t+\bar h;z,r) \right |^{\delta},
  \end{align*}
  where $\bar h \in (0,h)$. Using \cref{lma:heat-bounds} leads to
  \begin{align*}
    |\hat \Gamma(x,t+h;z,r) - \hat \Gamma(x,t;z,r)| \leq C h^\delta (t-r)^{-\delta}  e^{-\frac{\delta |z-x|^2}{C (t+h-r)}}.
  \end{align*}
  Since we also have (see~\cite{sanz2002holder})
  \begin{align*}
    |(t+h-r)^{\alpha-1-\frac{d}{2}} - (t-r)^{\alpha-1-\frac{d}{2}}| \leq &
    2 (t-r)^{(\alpha-1-\frac{d}{2})(1-\delta)} |(t+h-r)^{\alpha-1-\frac{d}{2}} - (t-r)^{\alpha-1-\frac{d}{2}}|^\delta
    \\
    \leq &
    C (t-r)^{(\alpha-1-\frac{d}{2})(1-\delta) + (\alpha-2-\frac{d}{2})\delta} h^\delta
    \\
    = & C (t-r)^{\alpha-1-\frac{d}{2} - \delta} h^\delta,
  \end{align*}
  putting all together gives
  \begin{align*}
    \int\limits_{\mathbb{R}^d} \left | \hat \Gamma(x,t+h;z,r) (t+h-r)^{\alpha-1-\frac{d}{2}} - \hat \Gamma(x,t;z,r) (t-r)^{\alpha-1-\frac{d}{2}} \right | dz
    \leq
    C h^\delta
  \end{align*}
  for $\delta \in (0,\alpha)$. Since $\alpha  <\nicefrac{1}{2} - \nicefrac{\beta}{4}$, this proves the second claim.
\end{proof}

\subsection{Technical lemmas}

\begin{lemma}\label{VolBound}
  Let $A,B\subset\mathbb{R}^d$ be Lebesgue measurable sets such that $|A|+|B|<\infty$
  and let $0<\beta<d$. Then there exists a constant $K_\beta$ such that
  \begin{equation}
    \sup_{\alpha>0}\sup_{c\in\mathbb{R}^d}\int\limits_{A}\int\limits_{B}\|z_1-\alpha z_2+c\|^{-\beta}dz_2dz_1\leq K_{\beta}(1\vee |A|)|B|.
  \end{equation}
\end{lemma}
\begin{proof}
  We have, for all $c\in\mathbb{R}^d$ and $\alpha>0$,
  \begin{multline*}
    \int\limits_{A}\int\limits_{B}\|z_1-\alpha z_2+c\|^{-\beta}dz_2dz_1
    \\
    =\int\limits_{B}\left(\int\limits_{A\cap\mathcal{B}_{\alpha z_2-c,1}}\|z_1- \alpha z_2+c\|^{-\beta}dz_1+\int\limits_{A\setminus\mathcal{B}_{\alpha z_2-c,1}}\|z_1- \alpha z_2+c\|^{-\beta}dz_1\right)dz_2,
  \end{multline*}
  where
  \[
    \int\limits_{A\cap\mathcal{B}_{\alpha z_2-c,1}}\|z_1- \alpha z_2+c\|^{-\beta}dz_1\leq\int\limits_{B_{0,1}}\frac{1}{\|z\|^\beta}<\frac{K_\beta}{2}
  \]
  and
  \[
    \int\limits_{A\setminus\mathcal{B}_{\alpha z_2-c,1}}\|z_1- \alpha z_2+c\|^{-\beta}dz_1\leq |A|.
  \]
  Thus,
  \begin{equation*}
    \int\limits_{A}\int\limits_{B}\|z_1-\alpha z_2+c\|^{-\beta}dz_2dz_1
    \leq \int\limits_B\left(\frac{K_{\beta}}{2}+|A|\right)dz\leq K_{\beta}|B|(1\vee |A|)
  \end{equation*}
  completing the proof.
\end{proof}
\begin{lemma}
  \label{lma:basic-estimates}
  Let
  \begin{align*}
    H(x,R) = \int\limits_{\mathcal{B}_{x,R}} \|y\| e^{-\frac{|y|^2}{C}} dy,
  \end{align*}
  and let $S_k = \mathcal{B}_{0,k+1} \setminus \mathcal{B}_{0,k}$ with $S_0 = \mathcal{B}_{0,1}$. Then there exists a constant $C$ such that, for $t \leq \sqrt{\varepsilon}$, we have
  \begin{align} \label{eq:H_scaling}
    \sum_{k=0}^{\infty}
    \sup_{S_{k}} H\left (\cdot,\frac{\varepsilon}{\sqrt{t}} \right ) \left(k_{1}^{d-1}\vee 1\right)
    \leq C \left (\frac{\varepsilon}{\sqrt{t}} \right )^d
  \end{align}
  and
  \begin{align} \label{eq:improved_scaling}
    \sum_{k=0}^{\infty}
    \sup_{S_{k}} e^{-\frac{(\|z\|-\varepsilon)^2}{2t}} \left(k_{1}^{d-1}\vee 1\right)
    \leq C \left (\frac{\varepsilon}{\sqrt{t}} \right )^{d-1}.
  \end{align}
\end{lemma}
\begin{proof}
  To prove \cref{eq:H_scaling}, we first write
  \begin{align*}
    \sum_{k=0}^{\infty}
    \sup_{S_{k}} H \left (\cdot,\frac{\varepsilon}{\sqrt{t}} \right ) \left(k_{1}^{d-1}\vee 1\right)
    =
    \sum_{k=0}^{\infty} \left(k_{1}^{d-1}\vee 1\right) \sup_{z \in S_{k}}\int\limits_{\mathcal{B}_{z,\varepsilon/\sqrt{t}}} \|y\| e^{-\frac{|y|^2}{C}} dy.
  \end{align*}
  If $k \leq \varepsilon/\sqrt{t}$, then $H \approx 1$. This gives
  \begin{multline*}
    \sum_{k=0}^{\lfloor\frac{\varepsilon}{\sqrt{t}}\rfloor} \left(k_{1}^{d-1}\vee 1\right) \sup_{z \in S_{k}}\int\limits_{\mathcal{B}_{z,\varepsilon/\sqrt{t}}} \|y\| e^{-\frac{|y|^2}{C}} dy
    +
    \sum_{k=\lfloor\frac{\varepsilon}{\sqrt{t}}\rfloor+1}^{\infty} \left(k_{1}^{d-1}\vee 1\right) \sup_{z \in S_{k}}\int\limits_{\mathcal{B}_{z,\varepsilon/\sqrt{t}}} \|y\| e^{-\frac{|y|^2}{C}} dy
    \\
    \leq C \left (\frac{\varepsilon}{\sqrt{t}} \right )^d + C.
  \end{multline*}
  To prove \cref{eq:improved_scaling}, we use a similar argument and obtain
  \begin{align*}
    \sum_{k=0}^{\infty}\sup_{z\in\ S_{k}}e^{-\frac{\left(\|z\|-\frac{\varepsilon}{\sqrt{t}}\right)^2}{2}}k^{d-1}
    =    & \sum_{k=0}^{\lfloor\frac{\varepsilon}{\sqrt{t}}\rfloor}\sup_{z\in\ S_{k}}e^{-\frac{\left(\|z\|-\frac{\varepsilon}{\sqrt{t}}\right)^2}{2}}\left(k^{d-1}\vee 1\right)+\sum_{k=\lfloor\frac{\varepsilon}{\sqrt{t}}\rfloor+1}^{\infty}2^{\frac{d-1}{2}}e^{-\frac{k^2}{8}} \\
    \leq & C\left(\left(2\frac{\varepsilon}{\sqrt{t}}\right)^{d-1}+1\right)
    \leq C \left (\frac{\varepsilon}{\sqrt{t}}\right)^{d-1}.
  \end{align*}
  This completes the proof.
\end{proof}

\section{Proof of \cref{main}}\label{sec:main-proof}
Throughout \cref{sec:main-proof}, we assume that $\dot{W}$ is white in time with spatial correlations given by $\gamma(x-y) =\Vert x-y\Vert^{-\beta}$.
\subsection{Introduction to the idea of the proof}
\label{subsec:heuristic}
The formal proof of our main theorem (given in \cref{subsec:main-proof,sec:A-proof,sec:R-proof}) is long and technical. In this subsection, we explain the main steps of the proof to make it easier for a reader to follow.

Let us consider a mild solution to the SPDE
\begin{align*}
  L u = \sigma(u) \dot{W}, \quad u(x,0) = 0,
\end{align*}
where $L$ is a linear operator and $\sigma$ is a non-linear function. Now the solution is given by
\begin{align*}
  u(x,t) = \int\limits_0^t \int\limits_{\R^d} \Gamma(x,t;y,s) \sigma(u(y,s)) W(dy,ds).
\end{align*}
By considering $Lu$ in the sense of distributions, this leads to (with $\langle \cdot, \cdot \rangle$ denoting the action of a distribution on a test function)
\begin{align*}
  \langle L u, \varphi \rangle = \int\limits_{\R^{d+1}} \int\limits_0^t \int\limits_{\R^d} \Gamma(x,t;y,s) \sigma(u(y,s)) W(dy,ds) L^\ast \varphi(x,t) dx dt,
\end{align*}
where $L^\ast$ is the formal adjoint operator of $L$. By interchanging the order of integration allows us to write
\begin{align*}
  \langle L u, \varphi \rangle = \int\limits_{\R^{d+1}} \left (\int\limits_s^\infty \int\limits_{\R^d} \Gamma(x,t;y,s) L^\ast \varphi(x,t) dx dt \right ) \sigma(u(y,s)) W(dy,ds).
\end{align*}
Here integration by parts gives, for any $\hat s > s$, that
\begin{align} \label{eq:main-players}
  \int\limits_{\hat s}^\infty \int\limits_{\R^d} \Gamma(x,t;y,s) L^\ast \varphi(x,t) dx dt =
  \int\limits_{\R^d} \Gamma(x,\hat s;y,s) \varphi(y,\hat s) dx + \int\limits_{\hat s}^\infty \int\limits_{\R^d} L \Gamma(x,t;y,s) \varphi(y,t) dx dt.
\end{align}
Since $L \Gamma(x,t;y,s) = 0$ for $t>s$, taking the limit $\hat s \to s$ leads to
\begin{align*}
  \int\limits_s^\infty \int\limits_{\R^d} \Gamma(x,t;y,s) L^\ast \varphi(x,t) dx dt =
  \varphi(y,s).
\end{align*}
Plugging this back into the Walsh integral yields
\begin{align*}
  \langle L u, \varphi \rangle = \int\limits_{\R^{d+1}} \varphi(y,s) \sigma(u(y,s)) W(dy,ds).
\end{align*}
Note that this is another way to justify the concept of a mild solution, as it gives $Lu = \sigma(u) \dot{W}$ in the sense of distributions if the action of $\sigma(u) \dot{W}$ on the test function $\varphi$ is interpreted in the Walsh sense. Next we note that if $\supp\{\varphi\} \subseteq \R^d \times [t_0,t]$, conditional isometry leads to
\begin{align*}
  \E\left [ |\langle L u, \varphi \rangle|^2 \mid \mathscr{F}_{t_0} \right ]
  =
  \int\limits_{t_0}^t \int\limits_{\R^d} \int\limits_{\R^d} \E[\sigma(u(z_1,r))\sigma(u(z_2,r))] \varphi(z_1,r) \varphi(z_2,r) \gamma(z_1-z_2) dz_1 dz_2 dr.
\end{align*}
In particular, setting $\varphi_{\varepsilon}(x,t) = \varphi(\frac{x-x_0}{\varepsilon}, \frac{t-t_0}{\varepsilon})$ as a sequence of bump functions that approximate the indicator function $\mathbb{I}_{\mathcal{B}_{x_0,\varepsilon}(x) \times (t_0,t_0+\varepsilon)}$ and defining the normalizing sequence
\begin{align*}
  m^2(\varepsilon) = \int\limits_{\R^{2d+1}} \varphi_\varepsilon(z_1,r) \varphi_\varepsilon(z_2,r) \gamma(z_1-z_2) dz_1 dz_2 dr \approx \varepsilon^{2d-\beta+1},
\end{align*}
we see that
\begin{align*}
  \lim_{\varepsilon \to 0} \E\left [ \left|\left\langle L u, \frac{\varphi_\varepsilon}{m(\varepsilon)}\right\rangle\right|^2 \mid \mathscr{F}_{t_0} \right ] = \sigma^2(u(x_0,t_0)).
\end{align*}
Our proof is based on the above idea, repeated for the discretized operator $L^h$ and using $\varphi_\varepsilon$ as a (non-smooth) indicator function. In particular, we need to control non-smoothness of the indicator at the boundary as well as non-smoothness of $\Gamma(x,t;y,s)$ at the diagonal $t=s$. The proof of our main theorem is divided into a main proof and separate propositions that cover these separate regions. We note that these non-smooth boundaries are the reason behind the assumption $\beta<1$.

\subsection{Decomposition and proof of \cref{main}}
\label{subsec:main-proof}
We begin by decomposing
\begin{align*}
  L^hu(y,s)= &
  \frac{1}{h^2}\left(u(y,s+h^2)-u(y,s)\right)-\sum_{i,j=1}^da_{ij}(y,s)(\mathcal D^{2,h}_{i,j}u)(y,s)+\sum_{i=1}^db_i(y,s)(\mathcal D^{1,h}_i u)(y,s) \\
  =          & A_{y,s,h}+R_{y,s,h}+\int\limits_{\mathbb{R}^d}L^h\Gamma(y,s;z,0)u_0(z)dz
\end{align*}
with
\begin{align}
  \label{eq:A}
  A_{y,s,h} & :=\int\limits_s^{s+h^2}\int\limits_{\mathbb{R}^d}\frac{\Gamma (y,s+h^2;z,r)}{h^2}\sigma(u(r,z))W(dz,dr) \quad \textrm{and} \\
  \label{eq:R}
  R_{y,s,h} & :=\int\limits_{0}^s\int\limits_{\mathbb{R}^d}(L^h \Gamma)(y,s;z,r)\sigma(u(r,z))W(dz,dr).
\end{align}

Now, by \cref{conditionalExp}, we have
\begin{equation*}
  \mathbb{E}_{u(x_0,t_0)}\left [\int\limits_{\mathcal B_{x_0,\varepsilon}}\int\limits_{\mathcal B_{x_0,\varepsilon}}\int\limits_{t_0}^{t_0+\varepsilon}\int\limits_{t_0}^{t_0+\varepsilon}A_{s_2,y_2,h}R_{s_1,y_1,h}ds_1ds_2dy_1dy_2\right ] = 0,
\end{equation*}
and this allows to write
\begin{multline*}
  \widehat{\sigma}^2_{\varepsilon,h}(u(x_0,t_0);x_0,t_0)
  =\mathbb{E}_{u(x_0,t_0)}\left [\left(\frac{1}{\widetilde{m}(\varepsilon,h)}\int\limits_{\mathcal B_{x_0,\varepsilon}}\int\limits_{t_0}^{t_0+\varepsilon}A_{y,s,h}dsdy\right)^2\right ]
  \\
  +\mathbb{E}_{u(x_0,t_0)}\left [\left(\frac{1}{\widetilde{m}(\varepsilon,h)}\int\limits_{\mathcal B_{x_0,\varepsilon}}\int\limits_{t_0}^{t_0+\varepsilon}R_{y,s,h}dsdy\right)^2\right ].
\end{multline*}
We next treat the terms $A_{y,s,h}$ and $R_{y,s,h}$ separately in \cref{prop:A} and \cref{prop:R} below.

In the sequel, if $\left(\mathbb{E}[|X_{h}-Y|^p]\right)^\frac{1}{p}\leq Cg(h)$
for a constant $C$ that is independent of $h$, we write $X_{h}=Y+O_{L^p}(g(h))$ for $p \in \mathbb{N}^*$ where $\mathbb{N}^* = \{1,2,\ldots\}$.

\begin{proposition}
  \label{prop:A}
  Let $\varepsilon=h^{\varrho}$ for some $\varrho\in (0,1)$ and let $A$ be given by \cref{eq:A}. Then, for all $(x_0,t_0)\in\mathbb{R}^d\times\mathbb{R}_+$, $p \in \mathbb{N}^*$ and for all $\nu \in \left(0,\frac{1}{2}-\frac{\beta}{4}\right)$, we have
  \begin{equation*}
    \mathbb{E}_{u(x_0,t_0)}\left [\left(\frac{1}{m(h)}\int\limits_{\mathcal B_{x_0,\varepsilon}}\int\limits_{t_0}^{t_0+\varepsilon}A_{y,s,h}dsdy\right)^2\right ]
    = \sigma^2(u(x_0,t_0))+O_{L^p}\left(h^{\varrho \nu}\right),
  \end{equation*}
  where the $m(h)$, given by  \cref{eq:m1}, satisfies
  \begin{align}
    \label{eq:normalising-bounds}
    \frac{h^{\varrho(2d-\beta+1)}}{C} \leq m^2(h) \leq C h^{\varrho(2d-\beta+1)}.
  \end{align}
\end{proposition}

\begin{proposition}
  \label{prop:R}
  Let $\varepsilon=h^{\varrho}$ for some $\varrho\in (0,1)$, let $\kappa\in (\varrho,1)$, $\beta \in (0,1)$, $R$ be given by \cref{eq:R}, and let $m(h)$ be given by  \cref{eq:m1}. Then, for all $(x_0,t_0)\in\mathbb{R}^d\times\mathbb{R}_+$ and $p \in \mathbb{N}^*$, we have
  \begin{equation*}
    \mathbb{E}_{u(x_0,t_0)}\left [\left(\frac{1}{m(h)}\int\limits_{\mathcal B_{x_0,\varepsilon}}\int\limits_{t_0}^{t_0+\varepsilon}R_{y,s,h}dsdy\right)^2\right ]
    = O_{L^p}\left(h^{2-2\kappa}+h^{(2-\beta)(\kappa-\varrho)}\right).
  \end{equation*}
\end{proposition}
The proof of \cref{prop:A} is given in \cref{sec:A-proof}, the proof of \cref{prop:R} is given in \cref{sec:R-proof}, and related auxiliary propositions and their proofs are given in \cref{sec:R-proof-aux}.

\begin{proof}[Proof of \cref{main}]
  \cref{main} follows directly by combining \cref{prop:A,prop:R} and optimizing the parameters $\varrho$, $\nu$, and $\kappa$. That is, we solve the problem
  \begin{align*}
    \arg\max_{\varrho,\nu,\kappa} \min\{\varrho \nu , 2-2\kappa, (2-\beta)(\kappa - \varrho)\}
  \end{align*}
  with the ranges provided in \cref{prop:A,prop:R}.
  Setting $\varrho \nu = 2-2\kappa$ gives $\varrho = \frac{2-2\kappa}{\nu}$, and then setting $(2-\beta)(\kappa - \varrho) = 2-2\kappa$ leads to
  \begin{align*}
    \kappa = \frac{2\nu+2(2-\beta)}{(2-\beta)(\nu+2)+2\nu}.
  \end{align*}
  Direct computations show that now $\kappa \in (\varrho,1)$ and $\kappa = \kappa(\nu)$ is a decreasing function. Thus, the convergence rate $\varrho\nu$ is maximized by choosing $\nu$ arbitrarily close to $\nu^*=\nicefrac{1}{2}-\nicefrac{\beta}{4}$, in which case $\kappa$ and $\varrho$ approach
  \begin{align*}
    \kappa^* = \frac{10}{12-\beta}, \quad \varrho^* = \frac{8}{12-\beta}.
  \end{align*}
  Then also $\varrho \in (0,1)$ and the rate becomes arbitrary close to $\varrho^*\nu^* =\frac{2(2-\beta)}{12-\beta}$. This completes the proof.
\end{proof}

\subsection{Proof of \cref{prop:A}}
\label{sec:A-proof}
\begin{proof}[Proof of \cref{prop:A}]
  By Fubini's theorem and \cref{lemma:conditional}, we have
  \begin{align*}
    \E_{u(x_0,t_0)} & \left [\left(\frac{1}{m(h)} \int\limits_{\mathcal{B}_{x_0,\varepsilon}}\int\limits_{t_0}^{t_0+\varepsilon}A_{y,s,h}dsdy\right)^2\right ]                                                                          \\
    & =\frac{1}{h^4 m^2(h)} \iiint\limits_{I_h(t_0) \times \mathbb{R}^{2d} }\mathbb{E}_{u(x_0,t_0)}\left [\sigma(u(r,z_1))\sigma(u(r,z_2))\right ] \widetilde{A}(z_1,r) \widetilde{A}(z_2,r)\gamma(z_1-z_2)dz_1 dz_2 dr \\
    & =\sigma^2(u(x_0,t_0))+\omega^1(h),
  \end{align*}
  with
  \begin{multline*}
    \omega^1(h) = \frac{1}{h^4 m^2(h)} \iiint\limits_{I_h(t_0) \times \mathbb{R}^{2d} } \left(\mathbb{E}_{u(x_0,t_0)}\left [\sigma(u(z_1,r))\sigma(u(z_2,r))-\sigma^2(u(x_0,t_0))\right ]\right)
    \\
    \times \widetilde{A}(z_1,r) \widetilde{A}(z_2,r)\gamma(z_1-z_2)dz_1 dz_2 dr
  \end{multline*}
  and $\widetilde{A}$ given by \cref{eq:tilde_A}.

  We next consider the bounds \cref{eq:normalising-bounds} for the normalizing sequence $m^2(h)$.
  For the upper bound, recall that by \cref{lma:heat-bounds} we have
  \begin{align*}
    \Gamma(x,t;\xi,\tau) \leq C\left(\frac{1}{(t-\tau)^{d/2}}\exp\left(-\frac{\|x-\xi\|^2}{C(t-\tau)}\right)\right) =: \Gamma_+(x-\xi,t-\tau).
  \end{align*}

  Let $\widetilde A_+(z,r)$ be defined as $\widetilde A(z,r)$ but with $\Gamma_+$ instead of $\Gamma$.
  Denoting $\mathcal{F}$ as the Fourier transform, we get by Plancherel's theorem that
  \begin{align*}
    m^2(h) & = \frac{1}{h^4} \iiint\limits_{I_h(t_0) \times \mathbb{R}^{2d} } \widetilde{A}_+(z_1,r) \widetilde{A}_+(z_2,r)\gamma(z_1-z_2)dz_1 dz_2 dr     \\
    & = C\frac{1}{h^4} \int\limits_{I_h(t_0)}\int\limits_{\mathbb{R}^{2d}} |\mathcal{F}{\widetilde{A}}_+(\xi,r)|^2 \mathcal{F} \gamma(\xi) d\xi dr,
  \end{align*}
  where
  \begin{align*}
    \mathcal{F}\widetilde{A}(\xi,r) = & \mathcal{F} \iint\limits_{I_{h,r}(t_0) \times \mathcal{B}_{x_0,\varepsilon}} \Gamma_+(y-z,s+h^2-r) ds dy \\
    = & \int\limits_{I_{h,r}(t_0)} \mathcal{F}(\Gamma_+(\cdot,s+h^2-r) \ast \mathbb{I}_{\mathcal{B}_{x_0,\varepsilon}})(\xi) ds.
  \end{align*}
  Here
  \begin{equation}
    \left \{
    \begin{aligned}
      \mathcal{F}(\Gamma_+(\cdot,\theta))(\xi) = & C \exp(-C \theta \|\xi\|^2),\\
      \mathcal{F}(\mathbb{I}_{\mathcal{B}_{x_0,\varepsilon}})(\xi) = & C \varepsilon^d \|\varepsilon \xi\|^{-\frac{d}{2}} J_{d/2}(\varepsilon \|\xi\|), \\
      \mathcal{F}\gamma = & C \|\xi\|^{\beta-d},
    \end{aligned}
    \right .
  \end{equation}
  where $J_{d/2}$ is the Bessel function of order $d/2$. Thus,
  \begin{align*}
    m^2(h) \leq & C \frac{\varepsilon^{2d}}{h^4} \int\limits_{I_h(t_0)}\,\iint\limits_{I_{h,r}^{\otimes 2}(t_0)}\,\int\limits_{\mathbb{R}^{d}} \exp(-C(s_1+s_2+2h^2-2r) \|\xi\|^2)  \hat \gamma(\xi) \|\varepsilon \xi\|^{-d} J_{d/2}^2(\varepsilon \|\xi\|) d\xi ds_1 ds_2 dr
    \\
    = & C \frac{\varepsilon^{d}}{h^4} \int\limits_{I_h(t_0)}\,\iint\limits_{I_{h,r}^{\otimes 2}(t_0)}\,\int\limits_{\mathbb{R}^{d}} \exp(-C(s_1+s_2+2h^2-2r) \|\xi/\varepsilon\|^2)  \hat \gamma(\xi/\varepsilon) \|\xi\|^{-d} J_{d/2}^2(\|\xi\|) d\xi ds_1 ds_2 dr
    \\
    \leq & C \frac{\varepsilon^{2d-\beta}}{h^4} \int\limits_{I_h(t_0)}\,\iint\limits_{I_{h,r}^{\otimes 2}(t_0)}\,\int\limits_{\mathbb{R}^{d}} \exp(-C(s_1+s_2+2h^2-2r) \|\xi/\varepsilon\|^2)  \|\xi\|^{\beta-2d} J_{d/2}^2(\|\xi\|) d\xi ds_1 ds_2 dr
    \\
    \leq &
    C \frac{\varepsilon^{2d-\beta}}{h^4 } \int\limits_{I_h(t_0)}\,\iint\limits_{I_{h,r}^{\otimes 2}(t_0)} ds_1 ds_2 dr = C \frac{\varepsilon^{2d-\beta}}{h^4 } \int\limits_{I_h(t_0)} (r \wedge (t_0+\varepsilon) - (r-h^2) \wedge t_0)^2 dr
    \\
    =& C\frac{\frac{2}{3} h^6 + h^4(\varepsilon - h^2)}{h^4} \varepsilon^{2d-\beta} = O(\varepsilon^{2d-\beta+1})
  \end{align*}
  proving the upper bound in \cref{eq:normalising-bounds} when plugging in  $\varepsilon = h^\varrho$.

  For the lower bound, we use the heat kernel lower bound for $\Gamma$ (see~\cite{Aronson}). That is, we have
  \begin{align*}
    \Gamma(x,t;\xi,\tau) \geq C\left(\frac{1}{(t-\tau)^{d/2}}\exp\left(-\frac{\|x-\xi\|^2}{C(t-\tau)}\right)\right) =: \Gamma_-(x-\xi,t-\tau).
  \end{align*}
  Proceeding as above and using the notation $\theta = s_1+s_2+2h^2-2r$ yields the lower bound
  \begin{align*}
    \varepsilon^{2d-\beta}\int\limits_{\mathbb{R}^{d}} \exp(-C\theta \|\xi/\varepsilon\|^2)  \|\xi\|^{\beta-2d} J_{d/2}^2(\|\xi\|) d\xi.
  \end{align*}
  Here
  \begin{align*}
    \int\limits_{\mathbb{R}^{d}} \exp(-C\theta \|\xi/\varepsilon\|^2)  \|\xi\|^{\beta-2d} J_{d/2}^2(\|\xi\|) d\xi
    \geq &
    \int\limits_{\mathcal{B}_{\varepsilon/\sqrt{\theta}}} \exp(-C\theta \|\xi/\varepsilon\|^2)  \|\xi\|^{\beta-2d} J_{d/2}^2(\|\xi\|) d\xi \\
    \geq &
    C \int\limits_{\mathcal{B}_{\varepsilon/\sqrt{\theta}}} \|\xi\|^{\beta-2d} J_{d/2}^2(\|\xi\|) d\xi,
  \end{align*}
  and since $\theta \leq 2h^2$, we have $\varepsilon/\sqrt{\theta} \geq \varepsilon^{\frac{\varrho-1}{\varrho}}$. Thus,
  \begin{equation*}
    \int\limits_{\mathcal{B}_{\varepsilon/\sqrt{\theta}}} \|\xi\|^{\beta-2d} J_{d/2}^2(\|\xi\|) d\xi \geq C
  \end{equation*}
  providing the lower bound in  \cref{eq:normalising-bounds}.

  It remains to bound the term $\omega^1(h)$.
  Recall that
  \begin{multline*}
    \omega^1(h) = \frac{1}{h^4 m^2(h)} \iiint\limits_{I_h(t_0) \times \mathbb{R}^{2d} } \left(\mathbb{E}_{u(x_0,t_0)}\left [\sigma(u(z_1,r))\sigma(u(z_2,r))-\sigma^2(u(x_0,t_0))\right ]\right)
    \\
    \times \widetilde{A}(z_1,r) \widetilde{A}(z_2,r)\gamma(z_1-z_2)dz_1 dz_2 dr.
  \end{multline*}
  Let $E_\sigma(z_1,z_2,r) = \sigma(u(z_1,r))\sigma(u(z_2,r))-\sigma^2(u(x_0,t_0))$ and let $\mathcal{K}(z_1,z_2,r) = \tilde A(z_1,r) \tilde A(z_2,r)\gamma(z_1-z_2)$.
  Then, by Fubini's theorem and Hölder's inequality, we get
  \begin{align*}
    \E[|\omega_1(h)|^p] = & \E \left [ \left | \mathbb{E}_{u(x_0,t_0)}\left [ \frac{1}{h^4 m^2(h)}\iiint\limits_{I_h(t_0) \times \mathbb{R}^{2d} } E_\sigma(z_1,z_2,r) \mathcal{K}(z_1,z_2,r) dz_1 dz_2 dr \right ] \right |^p \right ]
    \\
    \leq & \E \left [ \mathbb{E}_{u(x_0,t_0)}\left [ \left |\frac{1}{h^4 m^2(h)}\iiint\limits_{I_h(t_0) \times \mathbb{R}^{2d} } E_\sigma(z_1,z_2,r) \mathcal{K}(z_1,z_2,r) dz_1 dz_2 dr \right |^p \right ]  \right ]
    \\
    \leq & \E \left [ \left |\frac{1}{h^4 m^2(h)}\iiint\limits_{I_h(t_0) \times \mathbb{R}^{2d} } E_\sigma(z_1,z_2,r) \mathcal{K}(z_1,z_2,r) dz_1 dz_2 dr \right |^p \right ].
  \end{align*}
  Denoting $\Omega_h = I_h(t_0) \times \mathbb{R}^{2d}$, we can write
  \begin{multline} \label{eq:omega_1_bound}
    \E[|\omega_1(h)|^p] \leq \frac{1}{(h^4 m^2(h))^p} \int\limits_{\Omega_h} \cdots \int\limits_{\Omega_h} \E\left [\prod_{i=1}^p |E_\sigma(z_1^i,z_2^i,r_i)| \right ] \prod_{i=1}^p \mathcal{K}(z_1^i,z_2^i,r_i) \\
    \times dz_1^1\ldots dz_1^p dz_2^1\ldots dz_2^p dr_1\ldots dr_p.
  \end{multline}
  We now estimate the expectation in \cref{eq:omega_1_bound}. Now
  \begin{align*}
    \E\left [\prod_{i=1}^p |E_\sigma(z_1^i,z_2^i,r_i)| \right ]
    =
    \E\left [\prod_{i=1}^p (\sigma(z_1^i,r_i)-\sigma(x_0,t_0))(\sigma(z_2^i,r_i)+\sigma(x_0,t_0)) \right ].
  \end{align*}
  Using the inequality
  \begin{align*}
    ab-c^2 = (a-c)(b+c)+c(b-a) \leq |a-c||b+c| + |c||b-a| \leq |b+c|(|a-c| + |b-a|)
  \end{align*}
  together with Hölder's inequality yields
  \begin{multline*}
    \E \left [\prod_{i=1}^p (\sigma(z_1^i,r_i)-\sigma(x_0,t_0))(\sigma(z_2^i,r_i)+\sigma(x_0,t_0)) \right ]
    \\
    \begin{aligned}
      \leq &
      \E\left [\prod_{i=1}^p |\sigma(z_2^i,r_i)+\sigma(x_0,t_0)|(|\sigma(z_1^i,r_i)-\sigma(x_0,t_0)| + |\sigma(z_2^i,r_i)-\sigma(x_0,t_0)|) \right ]
      \\
      \leq &
      \prod_{i=1}^p \E\left [ |\sigma(z_2^i,r_i)+\sigma(x_0,t_0)|^p(|\sigma(z_1^i,r_i)-\sigma(x_0,t_0)| + |\sigma(z_2^i,r_i)-\sigma(x_0,t_0)|)^p \right ]^{\frac{1}{p}}
      \\
      \leq &
      \prod_{i=1}^p \E\left [ |\sigma(z_2^i,r_i)+\sigma(x_0,t_0)|^{2p} \right ]^{\frac{1}{2p}}
      \\
      & \times \left (\E \left [|\sigma(z_1^i,r_i)-\sigma(x_0,t_0)|^{2p}\right ]^{\frac{1}{2p}} + \E\left [|\sigma(z_2^i,r_i)-\sigma(x_0,t_0)|^{2p} \right ]^{\frac{1}{2p}} \right ).
    \end{aligned}
  \end{multline*}
  Using \cref{lem:moment_bounds_differences}, this leads to
  \begin{equation}\label{eq:product_moments}
    \E\left [\prod_{i=1}^p |E_\sigma(z_1^i,z_2^i,r_i)| \right ] \leq C_p \prod_{i=1}^p \left( \|z_1^i-x_0\|^{2\nu} + |r_i-t_0|^{\nu} + \|z_2^i-x_0\|^{2\nu} \right )
  \end{equation}
  with $\nu \in \left(0,\frac{2-\beta}{4}\right)$.
  Plugging \cref{eq:product_moments} into \cref{eq:omega_1_bound} gives an upper bound  of the form
  \begin{align}
    \int_{\Omega_h} & \mathcal{K}(z_1^i,z_2^i,r_i) (\|z_1^i-x_0\|^{2\nu} + |r_i-t_0|^{\nu} + \|z_2^i-x_0\|^{2\nu}) dz_1^i dz_2^i dr_i \notag
    \\
    = &
    \int_{\Omega_h} \tilde A(z_1^i,r_i) \tilde A(z_2^i,r_i) \gamma(z_1^i-z_2^i) (\|z_1^i-x_0\|^{2\nu} + |r_i-t_0|^{\nu} + \|z_2^i-x_0\|^{2\nu}) dz_1^i dz_2^i dr_i\notag
    \\
    = &
    \iiint\limits_{I_h(t_0) \times \mathbb{R}^{2d}}\,\, \iint\limits_{I_{h,r}^{\otimes 2}(t_0)}\,\,\iint\limits_{\mathcal{B}_{x_0,\varepsilon}^{\otimes 2}} \Gamma(y_1,s_1+h^2;z_1^i,r) \Gamma(y_2,s_2+h^2;z_2^i,r) \gamma(z_1^i-z_2^i)\notag \\
    & \times (\|z_1^i-x_0\|^{2\nu} + |r-t_0|^{\nu} + \|z_2^i-x_0\|^{2\nu}) dy_1 dy_2 ds_1 ds_2 dz_1^i dz_2^i dr. \label{eq:omega_1_bound2}
  \end{align}
  The heat kernel bounds now provide
  \begin{align}
    \Gamma(y_1,s_1+h^2;z_1,r) \|z_1-y_1\|^{2\nu}
    \leq & C \frac{1}{(s_1+h^2-r)^{d/2-\nu}} \exp\left(-\frac{\|y_1-z_1\|^2}{C(s_1+h^2-r)}\right) \notag
    \\
    =:& (s_1+h^2-r)^{\nu} \Gamma_+(y_1-z_1,s_1+h^2-r). \label{eq:scaling_bound}
  \end{align}

  For $y\in \mathcal{B}_{x_0,\varepsilon}$ we have $\|z-x_0\|\leq \|z-y\|+2\varepsilon$. Together with \cref{eq:scaling_bound}, this provides
  \begin{align*}
    \iiint\limits_{I_h(t_0) \times \mathbb{R}^{2d}}\,\, \iint\limits_{I_{h,r}^{\otimes 2}(t_0)}\,\,\iint\limits_{\mathcal{B}_{x_0,\varepsilon}^{\otimes 2}} & \Gamma(y_1,s_1+h^2;z_1^i,r) \Gamma(y_2,s_2+h^2;z_2^i,r) \gamma(z_1^i-z_2^i)\\
    & \times (\|z_1^i-x_0\|^{2\nu} + |r-t_0|^{\nu} + \|z_2^i-x_0\|^{2\nu}) dy_1 dy_2 ds_1 ds_2 dz_1^i dz_2^i dr
    \\
    \leq                                                                                                         & \iiint\limits_{I_h(t_0) \times \mathbb{R}^{2d}}\,\, \iint\limits_{I_{h,r}^{\otimes 2}(t_0)}\,\,\iint\limits_{\mathcal{B}_{x_0,\varepsilon}^{\otimes 2}} \Gamma_+(y_1-z_1^i,s_1+h^2-r) \Gamma_+(y_2-z_2^i,s_2+h^2-r) \gamma(z_1^i-z_2^i)\\
    & \times ((s_1+h^2-r)^\nu + 2\varepsilon^\nu + |r-t_0|^{\nu} + (s_2+h^2-r)^\nu) dy_1 dy_2 ds_1 ds_2 dz_1^i dz_2^i dr.
  \end{align*}
  Now, on the interval $I_{h,r}(t_0)$, we have
  $s_1+h^2-r \leq \min(r,t_0+\varepsilon)+h^2 - r \leq h^2$ and  $r-t_0 \leq \varepsilon$. This allows to obtain an upper bound
  \begin{equation}\label{eq:omega_1_bound_final}
    C (h^{2\nu} + \varepsilon^\nu) \iiint\limits_{I_h(t_0) \times \mathbb{R}^{2d} } \widetilde{A}_+(z_1^i,r) \widetilde{A}_+(z_2^i,r)\gamma(z_1^i-z_2^i)dz_1^i dz_2^i dr
    \leq C \varepsilon^\nu h^4 m^2(h).
  \end{equation}
  Combining the estimates  \cref{eq:omega_1_bound,eq:omega_1_bound2,eq:product_moments,eq:omega_1_bound_final} gives
  \begin{align*}
    \E[|\omega^1(h)|^p]^{\frac{1}{p}} \leq C_p h^{\varrho \nu}.
  \end{align*}
  This completes the proof.
\end{proof}

\subsection{Proof of \cref{prop:R}}
\label{sec:R-proof}
We begin the proof of \cref{prop:R} by stating three auxiliary propositions that are proven in \cref{sec:R-proof-aux}.

\begin{proposition}\label{Bound1}
  Let $\varepsilon = h^\varrho$ for some $\varrho\in(0,1)$, $\kappa\in (\varrho,1)$, and let $p\in\mathbb{N}^*$. Define $I^+_{h,r} = [t_0 \vee (r+h^{2\kappa})\wedge(t_0+\varepsilon),t_0+\varepsilon]$. Then, for all $(x_0,t_0)\in\mathbb{R}^d \times [0,T]$, we have that
  \begin{equation*}
    \mathbb{E}\left [\left|\int\limits_{0}^{t_0+\varepsilon}\int\limits_{\mathbb{R}^d}\left(\int\limits_{\mathcal{B}_{x_0,\varepsilon}}\int\limits_{I^+_{h,r}}\frac{(L^h\Gamma)(y,s;z,r)}{m(h)} dsdy\right)\sigma(u(x_0,t_0))W(dz,dr)\right|^{2p}\right ]^{\frac{1}{p}}
    \leq K h^{2-2\kappa}
  \end{equation*}
  for some constant $K$ depending solely on $T,\beta,\kappa,\varrho,M,x_0,p$ and $d$.
\end{proposition}

\begin{proposition}\label{Bound2}
  Let $\varepsilon = h^\varrho$ for some $\varrho\in(0,1)$, $\kappa\in (\varrho,1)$, $\beta \in (0,1)$, and let $p\in\mathbb{N}^*$. Define $I^-_{h,r} = [r\vee t_0,(r+h^{2\kappa})\wedge(t_0+\varepsilon)]$. Then, for all $(x_0,t_0)\in\mathbb{R}^d \times [0,T]$, we have that
  \begin{equation*}
    \mathbb{E}\left [\left|\,\int\limits_{t_0-h^{2\kappa}}^{t_0+\varepsilon}\int\limits_{\mathbb{R}^d}\left(\int\limits_{\mathcal{B}_{x_0,\varepsilon}}\int\limits_{I^-_{h,r}}
    \frac{(L^h\Gamma)(y,s;z,r)}{m(h)} dsdy\right)\sigma(u(x_0,t_0))W(dz,dr)\right|^{2p}\,\right ]^{\frac{1}{p}}
    \leq K h^{(\kappa-\varrho)(2-\beta)}
  \end{equation*}
  for some constant $K$ depending solely on $T,\beta,\kappa,\varrho,M,x_0,p$ and $d$.
\end{proposition}

\begin{proposition}\label{Bound3}
  Let $\varepsilon = h^\varrho$ for some $\varrho\in(0,1)$, $\kappa\in (\varrho,1)$, $\beta \in (0,1)$, and let $p\in\mathbb{N}^*$. Then, for any $\nu \in \left(0,\nicefrac{1}{2}-\nicefrac{\beta}{4}\right)$, we have that
  \begin{multline*}
    \mathbb{E}\left [\left|\int\limits_{0}^{t_0+\varepsilon}\int\limits_{\mathbb{R}^d}\left(\int\limits_{\mathcal{B}_{x_0,\varepsilon}}\int\limits_{t_0}^{t_0+\varepsilon}
    \frac{(L^h\Gamma)(y,s;z,r)}{m(h)}dsdy\right)\left(\sigma(u(x_0,t_0))-\sigma(u(r,z))\right)W(dz,dr)\right|^{2p}\right ]^{\frac{1}{p}}\\
    \leq K h^{\varrho \nu} (h^{2-2\kappa} + h^{(\kappa-\varrho)(2-\beta)})
  \end{multline*}
  for some constant $K$ depending solely on $T,\beta,\kappa,\varrho,M,x_0,p$, $d$, and $\nu$.
\end{proposition}

\begin{proof}[Proof of \cref{prop:R}]
  Let $\kappa\in (\varrho,1)$. We have that
  \begin{align*}
    \E \left [ \left \lvert \mathbb{E}_{u(x_0,t_0)}\left [\left(\int\limits_{\mathcal B_{x_0,\varepsilon}}\int\limits_{t_0}^{t_0+\varepsilon}R_{y,s,h}dsdy\right)^2\right ]\right \rvert^p \right ]^{\frac{1}{p}}
    \leq 3(R_1+R_2+R_3)
  \end{align*}
  with
  \begin{align*}
    R_1= & \mathbb{E}\left [\left|\int\limits_{0}^{t_0+\varepsilon}\int\limits_{\mathbb{R}^d}\left(\int\limits_{\mathcal{B}_{x_0,\varepsilon}}\int\limits_{(t_0\vee (r+h^{2\kappa}))\wedge(t_0+\varepsilon)}^{t_0+\varepsilon}
    \frac{(L^h\Gamma)(y,s;z,r)}{m(h)}dsdy\right)\sigma(u(x_0,t_0))W(dz,dr)\right|^{2p}\right ]^\frac{1}{p},                                                                                                         \\
    R_2= & \mathbb{E}\left [\left|\int\limits_{t_0-h^{2\kappa}}^{t_0+\varepsilon}\int\limits_{\mathbb{R}^d}\left(\int\limits_{\mathcal{B}_{x_0,\varepsilon}}\int\limits_{r\vee t_0}^{(r+h^{2\kappa})\wedge(t_0+\varepsilon)}
    \frac{(L^h\Gamma)(y,s;z,r)}{m(h)}dsdy\right)\sigma(u(x_0,t_0))W(dz,dr)\right|^{2p}\right ]^\frac{1}{p},
  \end{align*}
  and
  \begin{align*}
    R_3= & \mathbb{E}\left [\left|\int\limits_{0}^{t_0+\varepsilon}\int\limits_{\mathbb{R}^d}\left(\int\limits_{\mathcal{B}_{x_0,\varepsilon}}\int\limits_{t_0}^{(t_0+\varepsilon)}
    \frac{(L^h\Gamma)(y,s;z,r)}{m(h)}dsdy\right)\left(\sigma(u(x_0,t_0))-\sigma(u(r,z))\right)W(dz,dr)\right|^{2p}\right ]^\frac{1}{p}.\notag
  \end{align*}
  Applying \cref{Bound1} to $R_1$, \cref{Bound2} to $R_2$, and \cref{Bound3} to $R_3$ concludes the proof.

\end{proof}

\subsubsection{Proof of the auxiliary \cref{Bound1,Bound2,Bound3} related to \cref{prop:R}}
\label{sec:R-proof-aux}

\begin{proof}[Proof of \cref{Bound1}]
  Using \cref{eq:BDG} and the fact that
  \begin{equation*}
    \mathbb{E}[|\sigma^2(u(x_0,t_0))|^p]\leq C 
  \end{equation*}
  allows us to, without loss of generality, consider only the case $p=1$. Using Taylor expansion and the fact $L\Gamma = 0$ outside the diagonal $t = s$, we obtain
  \begin{multline*}
    (L^h\Gamma)(y,s;z,r)=h^2\int\limits_0^1(1-w)\partial_{tt}\Gamma(y,s+wh^2;z,r)dw\\
    +h\sum_{i,j=1}^d\sum_{k+l=3}\frac{3}{k!l!}a_{ij}(y,s)\int\limits_0^1(1-w)^2\partial_{x^k_ix^l_j}\Gamma(y+wh(e_i+e_j),s;z,r)dw\\
    +h\sum_{i=1}^d b_i(y,s)\int\limits_0^1(1-w)\partial_{x^2_i}\Gamma(y+whe_i,s;z,r)dw.
  \end{multline*}
  Consequently, the heat kernel bounds provided in \cref{lma:heat-bounds} give
  \begin{multline*}
    |(L^h\Gamma)(y,s;z,r)|
    \leq C\left(\frac{h^2}{(s-r)^2}+\frac{h}{(s-r)^\frac{3}{2}}\right)\int\limits_{0}^1 \Gamma_+(y-z; s-r+wh^2)
    \\
    +\Gamma_+(y+wh(e_i+e_j)-z;s-r) + \Gamma_+(y+whe_i-z;s-r)dw.
  \end{multline*}
  By the semigroup property of the heat kernel, we obtain
  \begin{multline*}
    \int\limits_{\mathbb{R}^d}\int\limits_{\mathbb{R}^d}|(L^h\Gamma)(y_1,s_1;z_1,r)||(L^h\Gamma)(y_2,s_2;z_2,r)|\gamma(z_1-z_2)dz_1dz_2\\
    \leq C G(y_1-y_2,s_1+s_2-2r)
    \times\left(\frac{h^4}{(s_1-r)^2(s_2-r)^2}+\frac{h^3}{(s_1-r)^2(s_2-r)^\frac{3}{2}}+\frac{h^2}{(s_1-r)^\frac{3}{2}(s_2-r)^\frac{3}{2}}\right),
  \end{multline*}
  where
  \begin{align*}
    G(x,\tau):= &
    \sum_{ij} \iint\limits_{[0,1]^2} \int\limits_{\mathbb{R}^d}\Gamma_+(x+z+w_1 \varphi_{i} + w_2 \varphi_{j};\tau)\gamma(z)dz dw_1 dw_2
    \\
    & +
    \sum_{i} \iint\limits_{[0,1]^2} \int\limits_{\mathbb{R}^d}\Gamma_+(x+z+w_1 \varphi_{i};\tau+w_2 h^2)\gamma(z)dz dw_1 dw_2
    \\
    & +
    \iint\limits_{[0,1]^2} \int\limits_{\mathbb{R}^d}\Gamma_+(x+z;\tau+w_1 h^2+w_2 h^2)\gamma(z)dz dw_1 dw_2
  \end{align*}
  for some vectors $\varphi_i \in \R^d$.
  Thus, by \cref{lemma:conditional}, we have
  \begin{multline*}
    \mathbb{E}\left [\left(\int\limits_{0}^{t_0+\varepsilon}\int\limits_{\mathbb{R}^d}\left(\int\limits_{\mathcal{B}_{x_0,\varepsilon}}\int\limits_{I^+_{h,r}}
    (L^h\Gamma)(y,s;z,r)dsdy\right)\sigma(u(x_0,t_0))W(dz,dr)\right)^2\right ]\\
    \leq C\mathbb{E}[\sigma^2(u(x_0,t_0))]\int\limits_0^{t_0+\varepsilon}\iint\limits_{(I^+_{h,r})^{\otimes 2}} \left(\frac{h^4}{(s_1-r)^2(s_2-r)^2}+\frac{h^3}{(s_1-r)^2(s_2-r)^\frac{3}{2}}+\frac{h^2}{(s_1-r)^\frac{3}{2}(s_2-r)^\frac{3}{2}}\right)\\
    \times\int\limits_{\mathcal{B}_{x_0,\varepsilon}}\int\limits_{\mathcal{B}_{x_0,\varepsilon}}
    G(y_1-y_2,s_1+s_2-2r) dy_1dy_2ds_1ds_2dr.
  \end{multline*}
  Proceeding as in the proof of \cref{prop:A} gives
  \begin{multline*}
    \int\limits_{\mathcal{B}_{x_0,\varepsilon}}\int\limits_{\mathcal{B}_{x_0,\varepsilon}}
    G(y_1-y_2,s_1+s_2-2r) dy_1dy_2
    \\
    \begin{aligned}
      \leq &
      C\sup_{\delta\in [0,2h^2],\psi\in\mathcal{B}_{0,4h}} \int\limits_{\mathbb{R}^d}|\mathcal{F} \Gamma_+(\xi+\psi;s_1+s_2-2r+\delta)| |\mathcal{F}(\mathbb{I}_{\mathcal{B}_{x_0,\varepsilon}})(\xi)|^2 \mathcal{F}(\gamma)(\xi)d\xi
      \\
      \leq & C \varepsilon^{2d-\beta} \int\limits_{\mathbb{R}^d} e^{-C(s_1+s_2-2r)^2 \|\xi/\varepsilon\|^2} \hat \gamma(\xi) \|\xi\|^{-d} J_{d/2}^2(\|\xi\|) d\xi
      \\
      \leq & C \varepsilon^{2d-\beta}.
    \end{aligned}
  \end{multline*}
  Consequently,
  \begin{multline*}
    \mathbb{E}\left [\left(\int\limits_{0}^{t_0+\varepsilon}\int\limits_{\mathbb{R}^d}\left(\int\limits_{\mathcal{B}_{x_0,\varepsilon}}\int\limits_{I^+_{h,r}}
    (L^h\Gamma)(y,s;z,r)dsdy\right)\sigma(u(x_0,t_0))W(dz,dr)\right)^2\right ]
    \\
    \begin{aligned}
      \leq & C \varepsilon^{2d-\beta}\int\limits_0^{t_0+\varepsilon}\iint\limits_{(I^+_{h,r})^{\otimes 2}}
      \left(\frac{h^4}{(s_1-r)^2(s_2-r)^2}+\frac{h^3}{(s_1-r)^2(s_2-r)^\frac{3}{2}}+\frac{h^2}{(s_1-r)^\frac{3}{2}(s_2-r)^\frac{3}{2}}\right) ds_1ds_2dr
      \\
      \leq & C \varepsilon^{2d-\beta}\left(h^4(\varepsilon h^{-4\kappa}+\varepsilon^{-2})+ h^3(\varepsilon h^{-3\kappa}+\varepsilon^{-\frac{3}{2}})+h^2(\varepsilon h^{-2\kappa}+\varepsilon^{-1})\right) \\
      \leq & C \varepsilon^{2d-\beta+1} \left(h^{4-4\kappa}+\frac{h^4}{\varepsilon^3} + h^{3-3\kappa}+\frac{h^3}{\varepsilon^{1+\frac32}} + h^{2-2\kappa}+\frac{h^2}{\varepsilon^2}\right).
    \end{aligned}
  \end{multline*}
  Plugging in $\varepsilon =h^{\varrho}$ and using
  $
    h^{4-4\kappa} \leq Ch^{3-3\kappa} \leq Ch^{2-2\kappa}
  $
  and
  $
    h^{4-3\varrho} \leq Ch^{2-2\varrho}
  $
  together with \cref{prop:A} yields the result.
\end{proof}

\begin{proof}[Proof of \cref{Bound2}]
  Again, without loss of generality, it suffices to consider the case $p=1$.  We start the proof with preliminary reductions. Recall that $I_h = [t_0-h^{2\kappa},t_0+\varepsilon]$ and $I^-_{h,r} = [r\vee t_0,(r+h^{2\kappa})\wedge(t_0+\varepsilon)]$. Now
  \begin{multline} \label{eq:prelim1}
    \mathbb{E}\left [\left(\int\limits_{I_h}\int\limits_{\mathbb{R}^d}\left(\int\limits_{\mathcal{B}_{x_0,\varepsilon}}\int\limits_{I^-_{h,r}}
    (L^h\Gamma)(y,s;z,r)dsdy\right)\sigma^2(u(x_0,t_0))W(dz,dr)\right)^2\right ] \\
    =
    \int\limits_{I_h} \iint\limits_{\R^{2d}} \iint\limits_{(I^-_{h,r})^{\otimes 2}} (F^{1,h}(\cdot;\cdot,r)+F^{2,h}(\cdot;\cdot,r))^{\otimes 2} \gamma(z_1-z_2)  ds_1 ds_2 dz_1 dz_2dr,
  \end{multline}
  where, with operators $T$ and $S$ defined in \cref{DefLh},
  \begin{equation} \label{eq:F1F2}
    \begin{aligned}
      F^{1,h}(s;z,r) = & \int\limits_{\mathcal{B}_{x_0,\varepsilon}}(\mathcal{D}_t^{h^2} \Gamma)(y,s;z,r)dy
      \\
      F^{2,h}(s;z,r) = & \int\limits_{\mathcal{B}_{x_0,\varepsilon}}(\mathcal{S}^h\Gamma)(y,s;z,r)dy.
    \end{aligned}
  \end{equation}

  By Taylor's theorem we can write
  \begin{align*}
    \mathcal{D}_t^{h^2} \Gamma(y,s;z,r) = & \int\limits_0^1 \partial_t\Gamma(y,s+wh^2;z,r) dw
  \end{align*}
  and
  \begin{align*}
    \mathcal{D}^{2,h}_{ij} f(x) = & \frac{f(x+h(e_i+e_j)) + f(x)-f(x+he_i) - f(x+he_j)}{h^2}
    \\
    = &
    \int\limits_0^1 (1-w) (e_i+e_j)^T D_x^2 f(x + h w (e_i+e_j)) (e_i+e_j) dw
    \\
    & -
    \int\limits_0^1 (1-w) e_i^T D_x^2 f(x + h w e_i) e_i dw
    \\
    & -
    \int\limits_0^1 (1-w) e_j^T D_x^2 f(x + h w e_j) e_j dw
    \\
    = & \sum_{k=1}^3 \sigma_k \int\limits_0^1 (1-w) \psi_{ij,k}^T D_x^2 f(x + h w \psi_{ij,k}) \psi_{ij,k} dw,
  \end{align*}
  where $\sigma_k \in \{-1,1\}$ and $D_x^2 f$ is the Hessian with respect to $x$.
  Thus,
  \begin{align*}
    a_{ij}(y,s) (\mathcal{D}^{2,h}_{ij} \Gamma)(y,s;z,r)
    =
    \sum_{k=1}^3 \sigma_k \int\limits_0^1 (1-w) a_{ij}(y,s) \psi_{ij,k}^T D_y^2 \Gamma(y + h w \psi_{ij,k},s;z,r) \psi_{ij,k} dw
  \end{align*}
  for some fixed vectors $\psi_{ij,k} \in \R^d$,
  and hence
  \begin{align*}
    \int\limits_{\mathcal{B}_{x_0,\varepsilon}} a_{ij}(y,s) (\mathcal{D}^{2,h}_{ij} \Gamma)(y,s;z,r) dy
    = &
    \sum_{k=1}^3 \sigma_k \int\limits_0^1 (1-w)
    \\
    & \times \int\limits_{\mathcal{B}_{x_0+hw\psi_{ij,k},\varepsilon}} a_{ij}(y-hw\psi_{ij,k},s) \psi_{ij,k}^T D_y^2 \Gamma(y,s;z,r) \psi_{ij,k} dy dw.
  \end{align*}
  We next estimate the integrals. We only consider the second order part in the operator $\mathcal{S}^h$ as the first order part in $\mathcal{S}^h$ is negligible, and the operator $\mathcal{D}_t^{h^2} $ concerning the time derivative can be treated similarly. For this, for a given vector $a \in \R^d$, we need to control integrals of the form
  \begin{align}
    \int\limits_{\mathcal{B}_{x_0+a,\varepsilon}} a_{ij}(y-a,s) \psi_{ij,k}^T D_y^2 \Gamma(y,s;z,r) \psi_{ij,k} dy
    =\int\limits_{\mathcal{B}_{x_0+a,\varepsilon}} a_{ij}(z-a,s) \psi_{ij,k}^T D_y^2 \Gamma(y,s;z,r) \psi_{ij,k} dy\nonumber
    \\
    + \int\limits_{\mathcal{B}_{x_0+a,\varepsilon}} (a_{ij}(y-a,s)-a_{ij}(z-a,s)) \psi_{ij,k}^T D_y^2 \Gamma(y,s;z,r) \psi_{ij,k} dy \label{eq:split}.
  \end{align}
  For the second term in \cref{eq:split}, Lipschitz continuity of $a_{ij}$ and the Gaussian bounds for $\Gamma$ gives
  \begin{align*}
    \int\limits_{\mathcal{B}_{x_0+a,\varepsilon}} (a_{ij}(y-a,s)-a_{ij}(z-a,s)) \psi_{ij,k}^T D_y^2 \Gamma(y,s;z,r) \psi_{ij,k} dy
    \leq
    C \int\limits_{\mathcal{B}_{x_0+a,\varepsilon}} \frac{\|y-z\|}{s-r}\Gamma_+(y-z; s-r) dy,
  \end{align*}
  where $\Gamma_+(y-z; s-r) = \frac{C}{(s-r)^{d/2}} e^{-\frac{|y-z|^2}{C(s-r)}}$.
  We next consider the first term in \cref{eq:split} separately in the cases  $z \in \mathcal{B}_{x_0+a,\varepsilon}^c$ and $z \in \mathcal{B}_{x_0+a,\varepsilon}$.
  If $z \in \mathcal{B}_{x_0+a,\varepsilon}$, then
  \begin{align*}
    (s-r)\int\limits_{\mathcal{B}_{x_0+a,\varepsilon}} \psi_{ij,k}^T D_y^2 \Gamma(y,s;z,r) \psi_{ij,k} dy
    =    &
    -(s-r)\int\limits_{\mathcal{B}_{x_0+a,\varepsilon}^c} \psi_{ij,k}^T D_y^2 \Gamma(y,s;z,r) \psi_{ij,k} dy
    \\
    \leq &
    C \int\limits_{\mathcal{B}_{x_0+a,\varepsilon}^c} \Gamma_+(y-z; s-r) dy
    \\
    \leq & \sup_{y\in \mathbb{R}^d\setminus\mathcal{B}_{x_0+a,\varepsilon}}e^{-\frac{\|z-y\|^2}{C(s-r)}}\leq e^{-\frac{(\varepsilon-\|z-x_0-a\|)^2}{C(s-r)}}.
  \end{align*}
  If $z \in \mathcal{B}_{x_0+a,\varepsilon}^c$, then
  \begin{align*}
    (s-r)\int\limits_{\mathcal{B}_{x_0+a,\varepsilon}} \psi_{ij,k}^T D_y^2 \Gamma(y,s;z,r) \psi_{ij,k} dy
    \leq &
    C \int\limits_{\mathcal{B}_{x_0+a,\varepsilon}} \Gamma_+(y-z; s-r) dy
    \\
    \leq & Ce^{-\frac{(\|z-x_0-a\|-\varepsilon)^2}{C(s-r)}}.
  \end{align*}
  It follows that in \cref{eq:prelim1} we obtain three separate terms (recall that first order term in $\mathcal{S}^h$ and $\mathcal{D}_t^{h^2} $ can be treated similarly). We first consider the cross-term arising from multiplying the first and the second term of \cref{eq:split}. This is considered using $U_1$ below. After that we consider the second powers of the first and second terms of \cref{eq:split}. These are considered using $U_2$ and $U_3$.

  In what follows, we use the notations of \cref{lma:basic-estimates}. Denote
  \begin{multline*}
    U_1 := \frac{1}{s_2-r}\iint\limits_{\R^{2d}}  \iint\limits_{[0,1]^2} \int\limits_{\mathcal{B}_{x_0+h w_1 \psi_k,\varepsilon}} \frac{\|y-z_1\|}{s_1-r}\Gamma_+(y-z_1; s_1-r) dy
    \\
    \times e^{-\frac{\Vert z_2-x_0-h w_2 \psi_l\Vert^2}{2(s_2-r)}} \gamma(z_1-z_2)  dw_1 dw_2dz_1 dz_2.
  \end{multline*}
  Assuming $s_1-r > s_2-r$ and making the change of variables $\hat z_1 = \frac{z_1}{\sqrt{s_1-r}}$, $\hat z_2 = \frac{z_2}{\sqrt{s_2-r}}$, $\hat y = \frac{y}{\sqrt{s_1-r}}$ allow us to write
  \begin{align*}
    U_1 = &
    (s_1-r)^{\frac{d}{2}-\frac{1}{2}-\frac{\beta}{2}} (s_2-r)^{\frac{d}{2}-1} \iint\limits_{[0,1]^2}\iint_{\R^{2d}}
    H \left (-\hat z_1, \frac{\varepsilon}{\sqrt{s_1-r}} \right ) e^{-\frac{|\hat z_2|^2}{2}}
    \\
    & \times \gamma \left (\hat z_1-\frac{\sqrt{s_2-r}}{\sqrt{s_1-r}} \hat z_2 -\frac{x_0+h w_1 \psi_k}{\sqrt{s_1-r}}-\sqrt{s_2-r}\frac{x_0+h w_2 \psi_l}{\sqrt{s_1-r}} \right )  d \hat z_1 d \hat z_2 dw_1 dw_2 \\
    = &
    (s_2-r)^{\frac{d}{2}-1}(s_1-r)^{\frac{d}{2}-\frac{1}{2}-\frac{\beta}{2}}\int\limits_0^1\int\limits_0^1
    \bigg ( \sum_{k_1=0}^{\infty}\sum_{k_2=0}^\infty
    \int\limits_{S_{k_1}} \int\limits_{S_{k_2}}
    H \left (-z_1,\frac{\varepsilon}{\sqrt{s_1-r}} \right ) e^{-\frac{|z_2|^2}{2}}
    \\
    & \times \left \|z_1-\frac{\sqrt{s_2-r}}{\sqrt{s_1-r}} z_2 -\frac{x_0+h w_2 \psi_k}{\sqrt{s_1-r}}-\sqrt{s_2-r}\frac{x_0+h w_1 \psi_l}{\sqrt{s_1-r}} \right \|^{-\beta} \bigg )  d z_1 d z_2 dw_1 dw_2
    \\
    \leq  &
    (s_2-r)^{\frac{d}{2}-1}(s_1-r)^{\frac{d}{2}-\frac{1}{2}-\frac{\beta}{2}}\int\limits_0^1\int\limits_0^1
    \\
    & \times \sum_{k_1=0}^{\infty}\sum_{k_2=0}^\infty
    \sup_{S_{k_1}} H\left (-\cdot,\frac{\varepsilon}{\sqrt{s_1-r}} \right ) \sup_{S_{k_2}} e^{-\frac{|\cdot|^2}{2}}
    \\
    & \times \int\limits_{S_{k_1}} \int\limits_{S_{k_2}} \left \|z_1-\frac{\sqrt{s_2-r}}{\sqrt{s_1-r}} z_2 -\frac{x_0+h w_2 \psi_k}{\sqrt{s_1-r}}-\sqrt{s_2-r}\frac{x_0+h w_1 \psi_l}{\sqrt{s_1-r}} \right \|^{-\beta}  d z_1 d z_2 dw_1 dw_2.
  \end{align*}
  \cref{VolBound} now gives
  \begin{align*}
    U_1
    \leq &
    (s_2-r)^{\frac{d}{2}-1}(s_1-r)^{\frac{d}{2}-\frac{1}{2}-\frac{\beta}{2}}
    \\
    & \times \sum_{k_1=0}^{\infty}\sum_{k_2=0}^\infty
    \sup_{S_{k_1}} H(-\cdot,\varepsilon,s_1-r) \sup_{S_{k_2}} e^{-\frac{|\cdot|^2}{2}} \left(k_{1}^{d-1}\vee 1\right)k_2^{d-1}.
  \end{align*}
  Consequently, using \cref{lma:basic-estimates}, leads to
  \begin{align*}
    U_1
    \leq & C
    (s_2-r)^{\frac{d}{2}-1}(s_1-r)^{\frac{d}{2}-\frac{1}{2}-\frac{\beta}{2}} \left ( \frac{\varepsilon}{\sqrt{s_1-r}} \right )^{d}\left ( \frac{\varepsilon}{\sqrt{s_2-r}} \right )^{d-1}
    \\
    = &
    C \varepsilon^{2d-1} (s_1-r)^{-\frac{1}{2}-\frac{\beta}{2}}(s_2-r)^{-\frac{1}{2}}.
  \end{align*}
  For the second power terms, we use similar arguments and obtain
  \begin{align*}
    U_2 := & \frac{1}{s_1-r}\frac{1}{s_2-r}\iint\limits_{[0,1]^2}\iint\limits_{\R^{2d}}  e^{-\frac{\Vert z_1-x_0-h w_1 \psi_k\Vert^2}{2(s_1-r)}} e^{-\frac{\Vert z_2-x_0-h w_2 \psi_l\Vert^2}{2(s_2-r)}} \gamma(z_1-z_2) dz_1 dz_2 dw_1 dw_2
    \\
    \leq & \varepsilon^{2d-2} (s_1-r)^{-\frac{1}{2}-\frac{\beta}{2}}(s_2-r)^{-\frac{1}{2}}
  \end{align*}
  and
  \begin{align*}
    U_3 := & \iint\limits_{[0,1]^2}\iint\limits_{\R^{2d}}  \int\limits_{\mathcal{B}_{x_0+h w_1 \psi_k,\varepsilon}} \frac{\|y-z_1\|}{s_1-r}\Gamma_+(y-z_1; s_1-r) dy
    \\
    & \times \int\limits_{\mathcal{B}_{x_0+h w_w \psi_l,\varepsilon}} \frac{\|y-z_2\|}{s_2-r}\Gamma_+(y-z_2; s_2-r) dy \gamma(z_1-z_2) dz_1 dz_2 dw_1 dw_2
    \\
    \leq & C \varepsilon^{2d} (s_1-r)^{-\frac{1}{2}-\frac{\beta}{2}}(s_2-r)^{-\frac{1}{2}}.
  \end{align*}

  In order to complete the proof, it remains to integrate in the time variable $r$. Using the bounds given above and the assumption $\beta<1$, we obtain
  \begin{multline*}
    \mathbb{E}\left [\left(\int\limits_{I_h}\int\limits_{\mathbb{R}^d}\left(\int\limits_{\mathcal{B}_{x_0,\varepsilon}}\int\limits_{I^-_{h,r}}
    (L^h\Gamma)(y,s;z,r)dsdy\right)\sigma^2(u(x_0,t_0))W(dz,dr)\right)^2\right ]
    \\
    \begin{aligned}
      & \leq C \int\limits_{I_h} \iint\limits_{(I^-_{h,r})^{\otimes 2}} (U_1+U_2+U_3) ds_1 ds_2 dr
      \\
      & \leq C \varepsilon^{2d-2} \int\limits_{I_h} \iint\limits_{(I^-_{h,r})^{\otimes 2}} (s_1-r)^{-\frac{1}{2}-\frac{\beta}{2}}(s_2-r)^{-\frac{1}{2}} ds_1 ds_2 dr
      \\
      & \leq C \varepsilon^{2d-2}(h^{\kappa(4-\beta)} + \varepsilon h^{\kappa(2-\beta)}).
    \end{aligned}
  \end{multline*}
  Dividing with $m^2(h) \sim \varepsilon^{2d-\beta+1}$ leads to
  \begin{align*}
    \varepsilon^{-2d+\beta-1} \varepsilon^{2d-2}(h^{\kappa(4-\beta)} + \varepsilon h^{\kappa(2-\beta)})
    = &
    h^{\varrho(\beta-4)+\varrho} (h^{\kappa(4-\beta)} + \varepsilon h^{\kappa(2-\beta)})
    =
    h^{(\kappa-\varrho)(4-\beta)+\varrho} (1 + h^{\varrho-2\kappa})
    \\
    = &
    h^{(\kappa-\varrho)(4-\beta)+\varrho} + h^{(\kappa-\varrho)(2-\beta)}
    \leq C h^{(\kappa-\varrho)(2-\beta)}.
  \end{align*}
  This completes the proof.
\end{proof}

\begin{proof}[Proof of \cref{Bound3}]
  Similarly as in bounding the term $\omega^1(h)$ in the proof of \cref{prop:A}, we use
  \begin{equation*}
    \mathbb{E}\left [|\sigma(u(z_1,r))\sigma(u(z_2,r))-\sigma^2(u(x_0,t_0))|\right ]
    \leq C ((\|z_1-y_1\|_2+\varepsilon)^{2\nu}+(\|z_2-y_2\|_2+\varepsilon)^{2\nu}+|r-t_0|^{\nu}),
  \end{equation*}
  valid for all $\nu \in \left(0,\nicefrac{1}{2}-\nicefrac{\beta}{4}\right)$. The claim is provided by using the properties of the kernel and following similar steps as in the proofs of \cref{Bound1,Bound2}.
\end{proof}

\section{Proof of \cref{main2}}
\label{sec:white-proof}
Throughout \cref{sec:white-proof}, we assume that $\dot{W}$ is the space-time white noise and $d=1$. The proof of \cref{main2} follows from the same arguments as the proof of \cref{main}. However, in this case, the isometry \cref{eq:covariance} is simply
\begin{equation*}
  \langle \varphi,\psi\rangle_{\gamma}:=\int\limits_{\mathbb{R}_+}\int\limits_{\mathbb{R}}\varphi(x,s)\psi(x,s)dxds
\end{equation*}
making the proof considerably easier. For this reason, we only present the main steps of the proof and leave the details to the reader.
\begin{proof}[Proof of \cref{main2}]
  Using the same steps as in the proof of \cref{main2}, we see that  the error consists of the following four terms
  \begin{align*}
    \omega^1(h) = \frac{1}{h^4 \hat{m}^2(h)} \iint\limits_{I_h(t_0) \times \mathbb{R} } \left(\mathbb{E}_{u(x_0,t_0)}\left [\sigma^2(u(z,r))-\sigma^2(u(x_0,t_0))\right ]\right)
    \widetilde{A}^2(z,r) dz dr,
  \end{align*}
  \begin{align*}
    R_1= & \mathbb{E}\left [\left|\int\limits_{0}^{t_0+\varepsilon}\int\limits_{\mathbb{R}}\left(\int\limits_{\mathcal{B}_{x_0,\varepsilon}}\int\limits_{(t_0\vee (r+h^{2\kappa}))\wedge(t_0+\varepsilon)}^{t_0+\varepsilon}
    \frac{(L^h\Gamma)(y,s;z,r)}{\hat{m}(h)}dsdy\right)\sigma(u(x_0,t_0))W(dz,dr)\right|^{2p}\right ]^\frac{1}{p},                                                                                                 \\
    R_2= & \mathbb{E}\left [\left|\int\limits_{t_0-h^{2\kappa}}^{t_0+\varepsilon}\int\limits_{\mathbb{R}}\left(\int\limits_{\mathcal{B}_{x_0,\varepsilon}}\int\limits_{r\vee t_0}^{(r+h^{2\kappa})\wedge(t_0+\varepsilon)}
    \frac{(L^h\Gamma)(y,s;z,r)}{\hat{m}(h)}dsdy\right)\sigma(u(x_0,t_0))W(dz,dr)\right|^{2p}\right ]^\frac{1}{p},
  \end{align*}
  and
  \begin{align*}
    R_3= & \mathbb{E}\left [\left|\int\limits_{0}^{t_0+\varepsilon}\int\limits_{\mathbb{R}}\left(\int\limits_{\mathcal{B}_{x_0,\varepsilon}}\int\limits_{t_0}^{(t_0+\varepsilon)}
    \frac{(L^h\Gamma)(y,s;z,r)}{\hat{m}(h)}dsdy\right)\left(\sigma(u(x_0,t_0))-\sigma(u(r,z))\right)W(dz,dr)\right|^{2p}\right ]^\frac{1}{p}.\notag
  \end{align*}
  Proceeding as in the proof of \cref{prop:A}, we obtain that the normalizing sequence satisfies
  \begin{align*}
    \frac{h^{2\varrho}}{C} \leq \hat{m}^2(h) \leq C h^{2\varrho}.
  \end{align*}
  To bound $\omega^1(h)$, we note that in this case we have bounds
  \begin{align*}
    \mathbb{E}\left [|I(x,t)-I(x+h,t)|^p\right ]
     & \leq C |h|^{2\nu p},
  \end{align*}
  and
  \begin{align*}
    \mathbb{E}\left [|I(x,t+h)-I(x,t)|^p\right ]
     & \leq C_T h^{\nu p}
  \end{align*}
  for any $\nu \in \left(0,\nicefrac{1}{4}\right)$. Following the same steps as in the proof of \cref{prop:A}
  leads to
  \begin{equation*}
    \mathbb{E}\left [|\omega^1(h)|^p\right ]^{\frac{1}{p}} \leq C_p \varepsilon^\nu = h^{\varrho\nu}
  \end{equation*}
  with $0<\nu < \nicefrac{1}{4}$. Similarly, proceeding
  as in the proof of \cref{Bound1} gives a bound
  \begin{equation*}
    R_1 \leq h^{2-2\kappa}.
  \end{equation*}
  For $R_2$, we follow the same steps as in the proof of \cref{Bound2} and obtain three terms given by
  \begin{equation*}
    U_1 := \frac{1}{s_2-r}\int\limits_{\R}  \iint\limits_{[0,1]^2} \int\limits_{\mathcal{B}_{x_0+h w_1 \psi_k,\varepsilon}} \frac{\|y-z\|}{s_1-r}\Gamma_+(y-z; s_1-r)
    e^{-\frac{\Vert z-x_0-h w_2 \psi_l\Vert^2}{2(s_2-r)}}   dydw_1 dw_2dz.
  \end{equation*}
  \begin{equation*}
    U_2 := \frac{1}{s_1-r}\frac{1}{s_2-r}\iint\limits_{[0,1]^2}\int\limits_{\R}  e^{-\frac{\Vert z-x_0-h w_1 \psi_k\Vert^2}{2(s_1-r)}} e^{-\frac{\Vert z-x_0-h w_2 \psi_l\Vert^2}{2(s_2-r)}}  dz dw_1 dw_2,
  \end{equation*}
  and
  \begin{multline*}
    U_3 := \iint\limits_{[0,1]^2}\int\limits_{\R}  \int\limits_{\mathcal{B}_{x_0+h w_1 \psi_k,\varepsilon}} \frac{\|y-z\|}{s_1-r}\Gamma_+(y-z; s_1-r) dy
    \\
    \times \int\limits_{\mathcal{B}_{x_0+h w_w \psi_l,\varepsilon}} \frac{\|y-z\|}{s_2-r}\Gamma_+(y-z; s_2-r) dy dz dw_1 dw_2.
  \end{multline*}
  Using the heat kernel bounds and the semigroup property, we deduce that
  \begin{equation*}
    U_1 \leq C\varepsilon (s_1-r)^{-\frac12}(s_2-r)^{-\frac12},
  \end{equation*}
  \begin{equation*}
    U_2 \leq C(s_1-r)^{-\frac12}(s_2-r)^{-\frac12},
  \end{equation*}
  and
  \begin{equation*}
    U_3 \leq C\varepsilon^2 (s_1-r)^{-\frac12}(s_2-r)^{-\frac12}.
  \end{equation*}
  From this, time integration then gives a bound
  \begin{equation*}
    R_2 \leq C\varepsilon^{-2}\left(h^{4\kappa}+\varepsilon h^{2\kappa}\right) \leq Ch^{2\kappa-\varrho}.
  \end{equation*}
  Finally, $R_3$ can be bounded by proceeding as in  the proof of \cref{Bound3}, and we obtain
  \begin{equation*}
    R_3 \leq C h^{\varrho \nu}\left(h^{2-2\kappa}+h^{2\kappa - \varrho}\right)
  \end{equation*}
  for $\nu \in\left(0,\nicefrac{1}{4}\right)$. As $\kappa>\varrho$, combining everything gives
  \begin{align*}
    \E \left [|\widehat{\sigma}^2_{\varepsilon, h}(u(x_0,t_0);x_0,t_0)-\sigma^2(u(x_0,t_0))|^p\right ]^{\frac{1}{p}}
     & \leq K \left(h^{\varrho\nu}+h^{2-2\kappa}+h^{2\kappa-\varrho}\right) \\
     & \leq K \left(h^{\varrho\nu}+h^{2-2\kappa}\right).
  \end{align*}
  The result now follows by choosing $\kappa \approx \varrho$ and optimizing the choices of $\varrho$ and $\nu$.
\end{proof}

\appendix

\section{Additional visualizations and discussion related to \cref{sec:numerical}}\label{sec:additional_plots}
In this section, we provide additional visualizations and discussion related to the simulations performed in \cref{sec:numerical}, in the particular case $\sigma(x) = \sigma_3(x) = \frac{1}{8} \exp(\sin(4|x-2|))$.

\cref{fig:calibration_curves} elaborates the effect of $\varepsilon$ and $h$ by plotting the true values of $\sigma_3$ against the estimated ones. The dashed blue lines in \cref{fig:calibration_curves} correspond to perfect estimation. \cref{fig:estimated_sd_residuals} displays, in addition to plotting the true values against the estimated values, also the estimated conditional standard deviations. We see that certain choices of $h$ and $\varepsilon$ have significantly better calibration. In addition, the behavior of the standard deviations in \cref{fig:estimated_sd_residuals} appears to depend on the true value of $\sigma_3$. Similar observations can be made from \cref{fig:plot_sigma_function_1,fig:plot_sigma_function_2,fig:grid_sigma_functions,fig:plot_sigma_functions_matrix,fig:grid_sigma_functions_epsilon_order}.

Our visualizations reveal that, as $\varepsilon$ increases, the peaks become flatter and move to the right. This could be explained by the skewness of the distribution of $u$. In particular, since we condition on $u(x_0,t_0)$, the integrated predictor uses future values. Provided that $u(x_0,t_0)$ is relatively large, these future values are typically lower than $u(x_0,t_0)$ due to the rapid decay of the heat kernel. This causes a shift/delay in the regression curve.

We use a simple stochastic model to verify this effect. We simulate the model setup as follows. We start with the points $x_1,\ldots,x_n$, and the true regression function $\sigma_3^2(x) = \frac{1}{64} \exp(2\sin(4|x-2|))$. Then we set
\begin{align*}
  y = \sigma_3^2(x)
\end{align*}
and
\begin{align*}
  \tilde x_i = x_i + 2 \sigma_3^2(y_i)\varepsilon_i,
\end{align*}
where $1-\varepsilon_i \sim \chi^2(1)$. In this model we have noisy measurements $\tilde{x}_i$, with the skewed noise proportional to $\sigma_3^2(y_i)$ resembling the above phenomena. \cref{fig:shifting} that displays the kernel regression results for the simple stochastic model confirms our reasoning.

\begin{table}[ht]
  \centering
  \begin{tabular}{c|ccc}
    $\varepsilon \backslash h$ & $\frac{2}{N_x}$ & $\frac{4}{N_x}$ & $\frac{8}{N_x}$ \\ \hline
    $\frac{4}{N_x}$            & \includegraphicspage[0.29\textwidth]{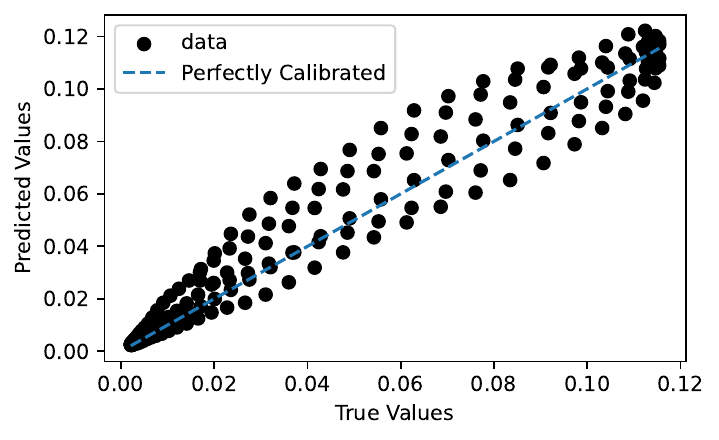} & \includegraphicspage[0.29\textwidth]{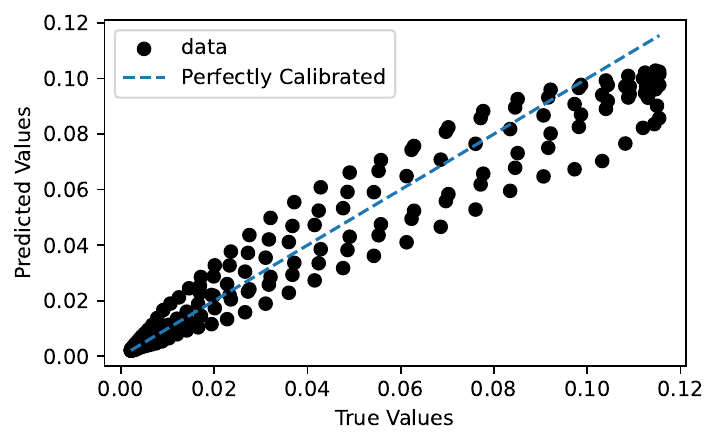} & \\
    $\frac{8}{N_x}$            & \includegraphicspage[0.29\textwidth]{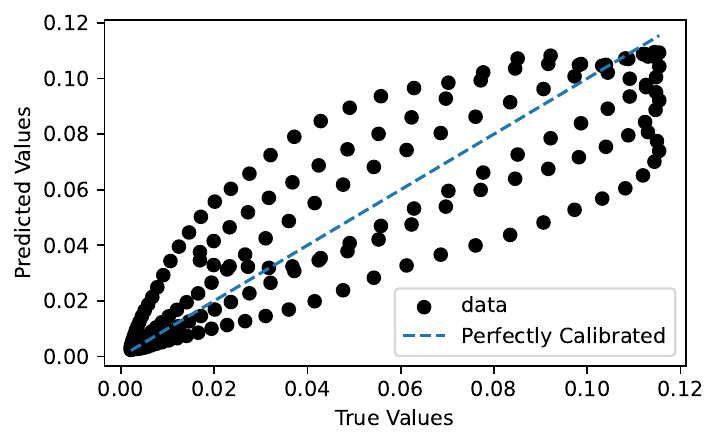} & \includegraphicspage[0.29\textwidth]{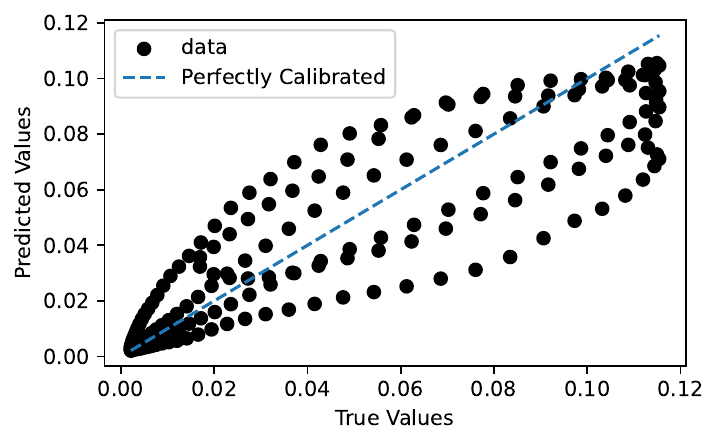} & \includegraphicspage[0.29\textwidth]{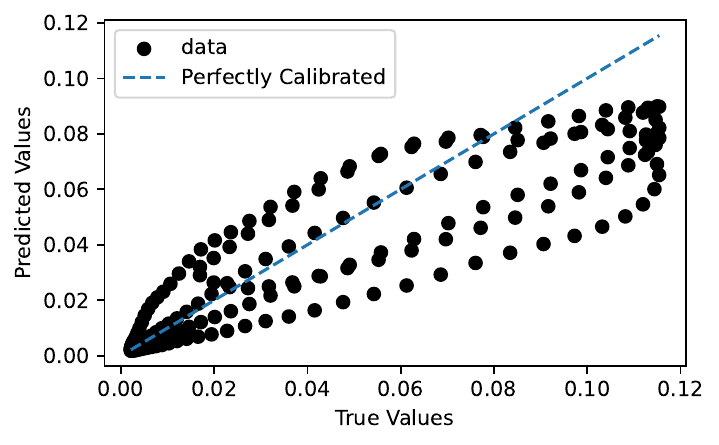} \\
    $\frac{16}{N_x}$           & \includegraphicspage[0.29\textwidth]{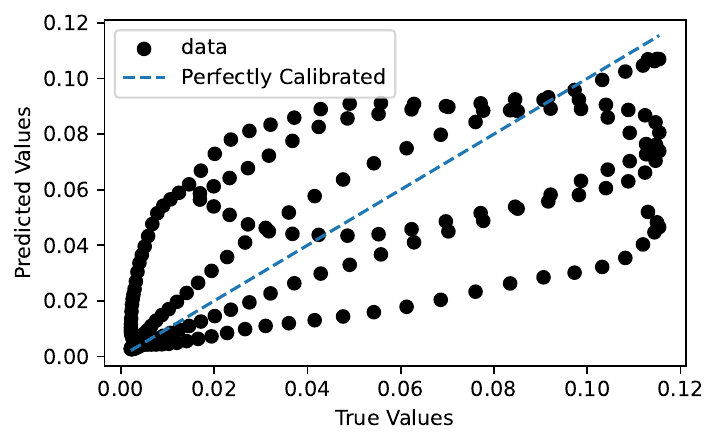} & \includegraphicspage[0.29\textwidth]{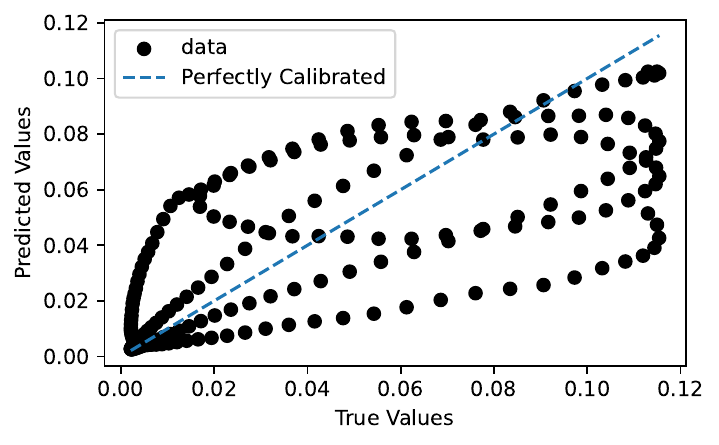} & \includegraphicspage[0.29\textwidth]{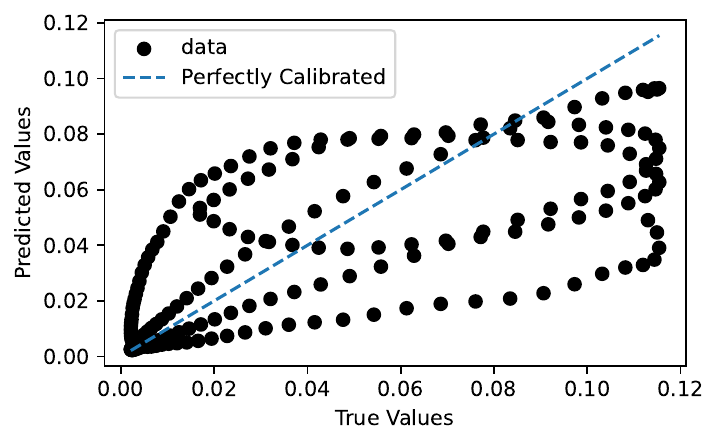} \\
    $\frac{32}{N_x}$           & & \includegraphicspage[0.29\textwidth]{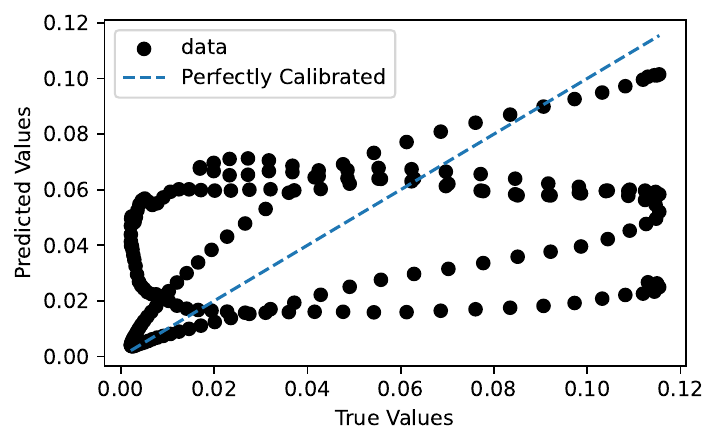} & \includegraphicspage[0.29\textwidth]{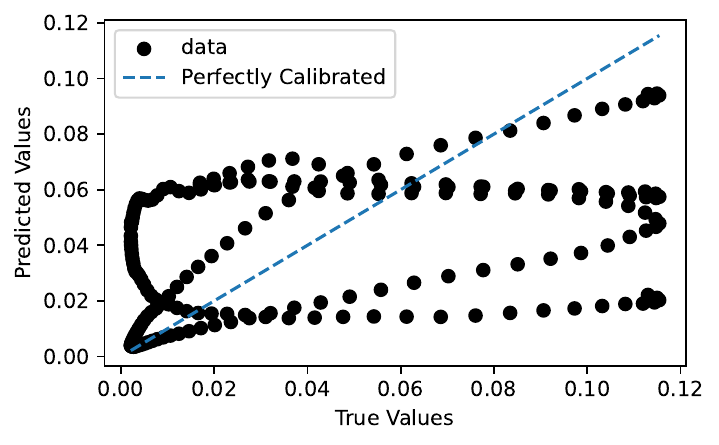} \\
  \end{tabular}
  \caption{True values $\sigma_3(x)$ against the estimated values $\hat\sigma_3(x)$ for fixed values of $\varepsilon$ (rows) and for fixed values of $h$ (columns).}
  \label{fig:calibration_curves}
\end{table}

\begin{table}[ht]
  \centering
  \begin{tabular}{c|ccc}
    $\varepsilon \backslash h$ & $\frac{2}{N_x}$ & $\frac{4}{N_x}$ & $\frac{8}{N_x}$ \\ \hline
    $\frac{4}{N_x}$            & \includegraphicspage[0.29\textwidth]{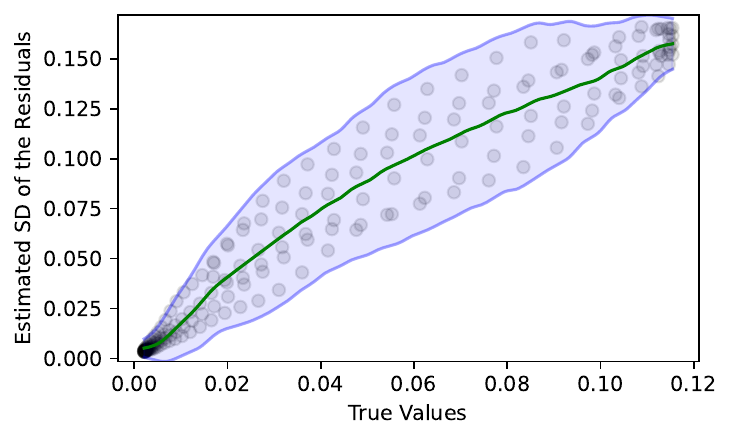} & \includegraphicspage[0.29\textwidth]{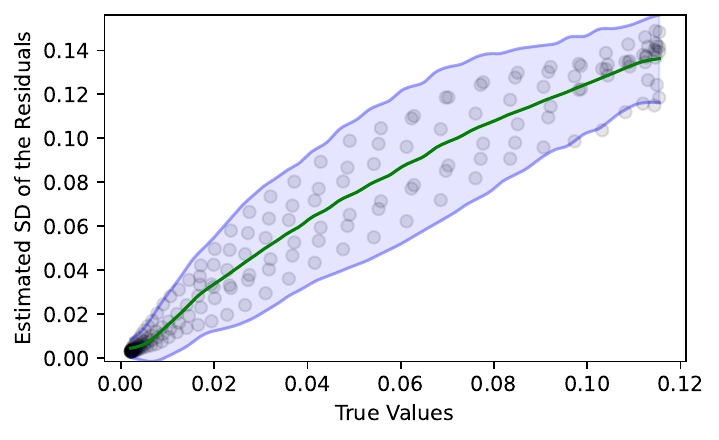} & \\
    $\frac{8}{N_x}$            & \includegraphicspage[0.29\textwidth]{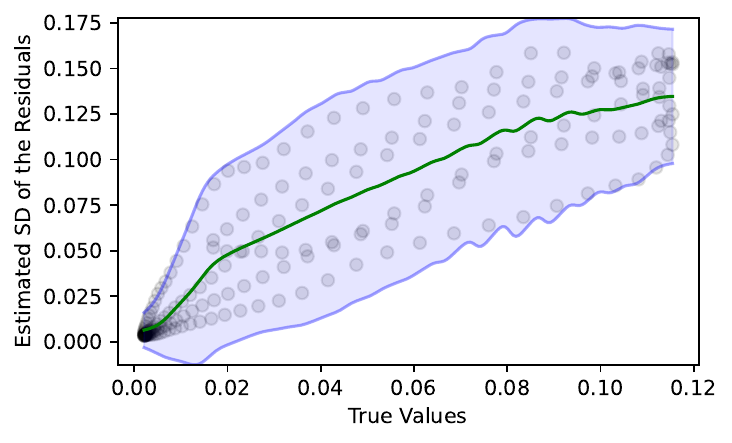} & \includegraphicspage[0.29\textwidth]{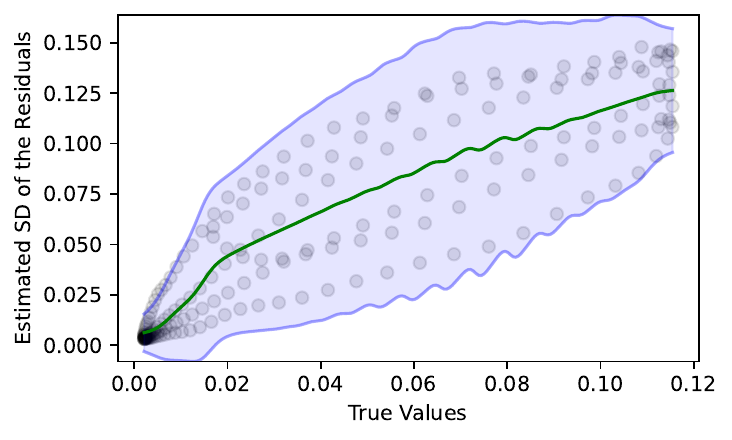} & \includegraphicspage[0.29\textwidth]{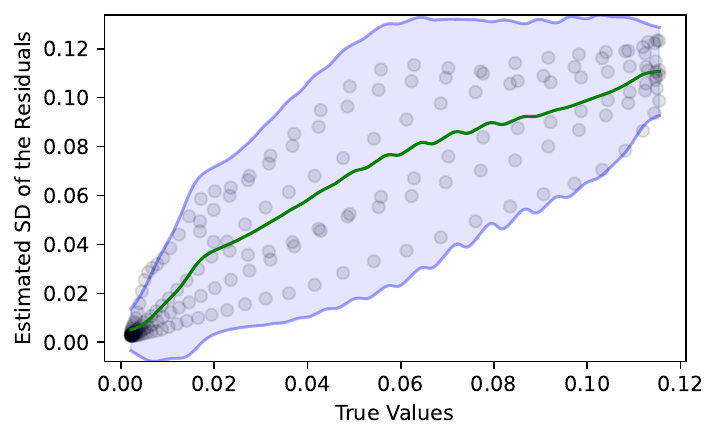} \\
    $\frac{16}{N_x}$           & \includegraphicspage[0.29\textwidth]{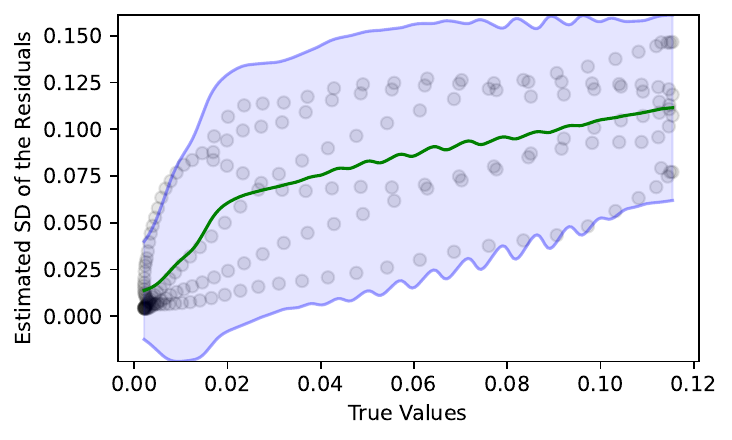} & \includegraphicspage[0.29\textwidth]{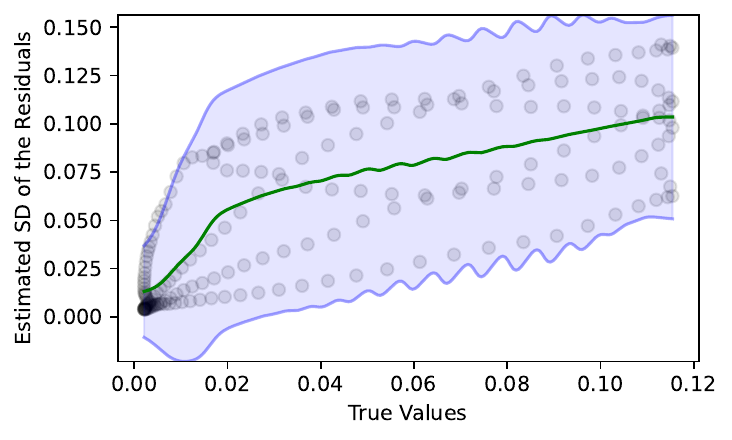} & \includegraphicspage[0.29\textwidth]{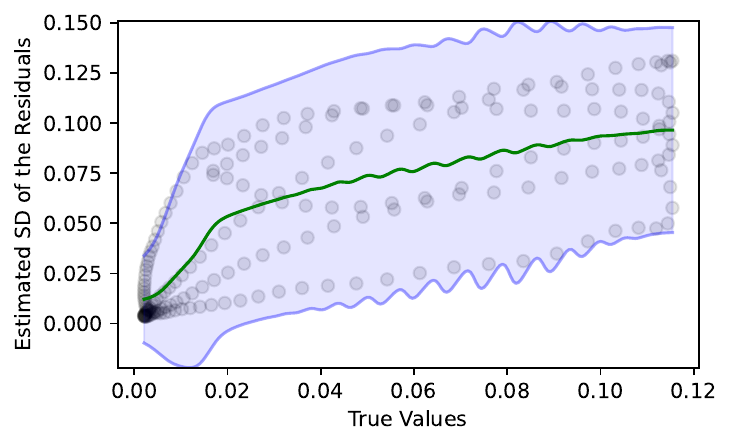} \\
    $\frac{32}{N_x}$           &                                                                                              & \includegraphicspage[0.29\textwidth]{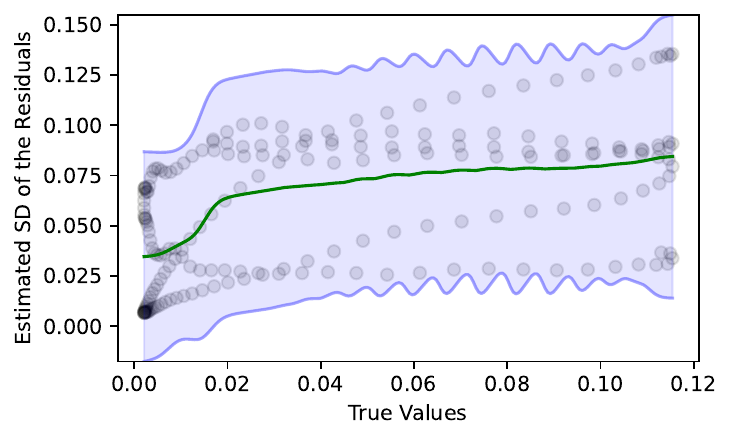} & \includegraphicspage[0.29\textwidth]{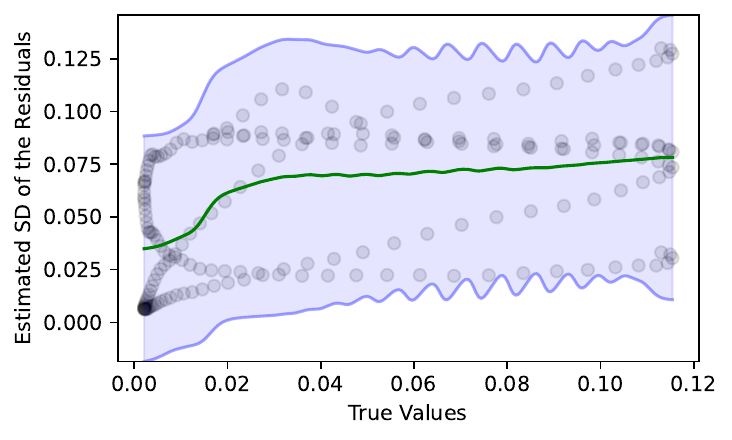} \\
  \end{tabular}
  \caption{True values $\sigma_3(x)$ against the estimated values $\hat\sigma_3(x)$ and estimated standard deviations for fixed values of $\varepsilon$ (rows) and for fixed values of $h$ (columns).}
  \label{fig:estimated_sd_residuals}
\end{table}

\begin{figure}[ht]
  \centering
  \includegraphics[width=0.7\textwidth]{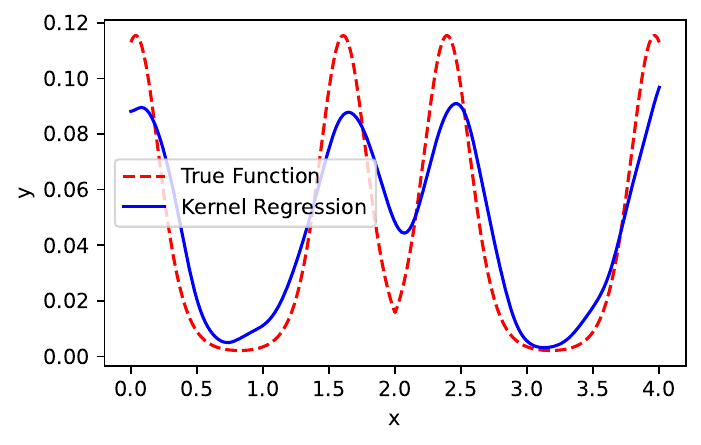}
  \caption{Illustration of the observed shifting phenomenon in the predictor.}
  \label{fig:shifting}
\end{figure}

\end{document}